\documentclass[11pt, reqno]{amsart}
\usepackage{amssymb}
\usepackage{verbatim}

\newtheorem{thm}{Theorem}[section]
\newtheorem{lemma}[thm]{Lemma}
\newtheorem{conjecture}[thm]{Conjecture}

\newtheorem{prop}[thm]{Proposition}

\newtheorem{defi}[thm]{Definition}

\theoremstyle{remark}

\newtheorem*{remark}{Remark}

\theoremstyle{definition}

\numberwithin{equation}{section}

\DeclareMathOperator{\Disc}{Disc}

\newcommand{\ZZ}{{\mathbb Z}}

\newcommand{\CC}{{\mathbb C}}

\newcommand{\PP}{{\mathbb {P}}}
\newcommand{\QQ}{{\mathbb {Q}}}
\newcommand{\RR}{{\mathbb {R}}}

\renewcommand{\hat}{\widehat}

\newcommand{\kk}{{k}}

\newcommand{\sF}{{\mathcal F}}

\newcommand{\sO}{{\mathcal O}}
\newcommand{\sS}{{\mathcal S}}
\newcommand{\FF} {{\mathbb F}}

\newcommand{\tns}{T_n}
\newcommand{\tnss}{T_n^{\ast}}

\newcommand{\ssy}{\mu}
\newcommand{\SSY}{{\mathfrak U}}
\newcommand{\pram}{{\rho}}
\newcommand{\nssy}{\nu_{n}(C_{\ssy})}
\newcommand{\cssy}{{C_{\ssy}}}
\newcommand{\fp}{{\mathfrak p}}

\newcommand{\sgn}{{\rm Sign}}

\newcommand \nip{{M_i(p)}}

\newcommand \nmx{{M_m(X)}}

\newcommand{\beql}[1]{\begin{equation}\label{#1}}
\newcommand{\eeq}{\end{equation}}

\begin{document}


\title{Splitting Behavior of $S_n$-Polynomials}

\author{Jeffrey C.  Lagarias}
\address{Dept. of Mathematics\\
University of Michigan \\
Ann Arbor, MI 48109-1043\\
}
\email{lagarias@umich.edu}

\author{Benjamin L. Weiss}
\address{Dept. of Mathematics\\ University of Maine\\
Orono, ME 04469\\}
\email{weiss@math.umaine.edu}


\thanks{The first author was partially supported by NSF grants  DMS-1101373 and DMS-1401224}
 \date{March 13, 2015}

\begin{abstract} 
We  analyze the probability that, for a fixed finite set of primes $S$, a random, monic, degree $n$ polynomial  $f(x) \in \ZZ[x]$
with coefficients in  a box of side $B$  satisfies: 
(i) $f(x)$ is irreducible over $\QQ$, with 
splitting field $K_f/\QQ$ over $\QQ$ having  Galois group $S_n$;  
(ii) the  polynomial discriminant $Disc(f)$  is relatively prime to all primes in $S$; 
(iii) $f(x)$ has  a  prescribed splitting type $(\bmod \, p)$ at  each  prime $p$ in $S$.

The limit  probabilities as $B \to \infty$ are described in terms of 
values of  a one-parameter family of
measures on $S_n$, called  $z$-splitting measures, with parameter $z$ evaluated
at the primes $p$ in $S$. We study properties of these measures.
We  deduce that   there 
exist degree $n$  extensions of $\QQ$ with  
Galois closure having Galois group $S_n$
with a given finite set of primes $S$ having given Artin symbols, 
 with some restrictions on allowed Artin symbols for $p<n$.
 We compare the distributions of these measures
 with  distributions formulated by Bhargava for   splitting probabilities for
a fixed prime $p$ in such degree $n$ extensions ordered by size of
discriminant, conditioned to be relatively prime to $p$.

\end{abstract}

\maketitle


\section{Introduction}

 By an  {\em $S_n$-polynomial}   we mean   a degree $n$ monic polynomial $f(x) \in \ZZ[x]$
whose 
splitting field $K_f/\QQ$, obtained by adjoining all
roots of $f(x)$ has Galois group $S_n$.
 It is well known that with high probability  a ``random" degree $n$ monic polynomial with 
 integer coefficients independently drawn from  a box $[-B, B]^n$ is
 irreducible and is an $S_n$-polynomial. 
In 1936 van der Waerden \cite{Waerden:1936}
showed  that   this probability approaches $1$
as the box size $B \to \infty$.  
For such a polynomial, adjoining one root of $f(x)$ gives
an $S_n$-number field.  Later authors obtained quantitative versions giving
explicit bounds for the cardinality of  the exceptional set; see  Section \ref{sec41}.

This paper considers a refinement of this problem: to  study the set of 
polynomials with coefficients in a box $[-B, B]^n$ which are $S_n$-polynomials  
 prescribed to have a given splitting behavior at a given finite set of primes $\{p_k : 1 \le k \le r\}$. 
It shows the existence of  limiting splitting densities as $B \to \infty$, conditional on the discriminant $\Disc(f)$ of the polynomial $f$ being relatively prime to  $\prod_{i=1}^r  p_i$.
This conditioning imposes  a non-ramification condition, requiring the polynomials to have square-free factorizations $\pmod {p_i}$ with $1 \le i \le r$.   This conditioning has two important consequences: 
\begin{enumerate}
\item[(1)]
The square-free assumption permits the  limiting splitting densities  to be interpreted as  a set of probability distributions
on the symmetric group $S_n$, which depend on the  prime $p$. 
These distributions are  constant on conjugacy classes of $S_n$.
\item[(2)]
The resulting limit of distributions  possess an interpolation property as $p$ varies. 
The splitting densities are the values at $z=p$ of a one-parameter family of  complex-valued
measures $\nu_{n,z}^{\ast}$ on the symmetric group $S_n$ 
which we call {\em $z$-splitting measures.} 
The interpolation property is: {the values $\nu_{n, z}^{\ast}(g)$  on fixed elements $g \in S_n$ are  rational functions
in the  parameter $z$.} 
\end{enumerate}
These limiting splitting densities at $z=p$  have a simple origin. They are inherited  from corresponding
densities for splitting of polynomials in $p$-adic fields recently studied by the second author
\cite{Weiss:2013}, which in turn arise from splitting probabilities for polynomials over
finite fields. 
The latter probabilities  are evaluated by  counting  the  monic polynomials over $\FF_p$ having
various square-free factorization types in $\FF_p[X]$,  for which there are 
explicit combinatorial formulas. 

The first contribution of this paper is to introduce  and study the
 $z$-splitting measures on $S_n$, and show that for parameter values $z=p$ they
 are limiting splitting distributions for $S_n$-polynomials above as the box size $B \to \infty$.
A  second contribution is to 
to compare and contrast the limiting probabilities of the model of this paper to a recent probability
model of Bhargava \cite{Bhargava:2007}, 
which  considers algebraic number fields of degree $n$, called $S_n$-number fields, whose normal closure has Galois group $S_n$.
Bhargava's model  concerns limiting splitting probabilities of a fixed prime $p$ taken over 
$S_n$-number fields having discriminant bounded by a parameter $D$, as $D \to \infty$.
The  interesting feature is that  the limiting probabilities of the two models do not agree.
 We now describe these two contributions in more detail.


\subsection{Existence and properties of $z$-splitting  measures} \label{sec11}

The paper directly defines the $z$-splitting measures
as rational functions of $z$ by a  combinatorial formula given in Definition \ref{de21a}, and studies their basic properties in Section \ref{sec3}.
Only later in the paper do we show that for $z$ a prime power these $p^k$-splitting densities coincide with
the limiting densities for  splitting
of $S_n$-polynomials, doing this for $k=1$  in  Section \ref{sec4} over  the rational field $\QQ$ and  
for general $k$ in Section \ref{sec5}  for polynomials with coefficients
over  general number fields.

 The splitting types of a square-free monic polynomial $(\bmod \, p)$ of degree $n$ are described by partitions $\ssy$
of $n$, which are identified with conjugacy classes on the symmetric group $S_n$.  
For each $n \ge 1,$ and for each prime $p$ in Section \ref{sec21} we define {\em $p$-splitting measures}
$\nu_{n, p}^{\ast}(\cdot)$ on $S_n$ which are constant on conjugacy classes $\cssy$ of $S_n$.  
We show the following results, whose precise statements are given in Section \ref{sec2}.
\begin{enumerate}
\item[(i)]
For a fixed prime $p$, the limiting probabilities as $B \to \infty$ for degree $n$ monic polynomials $f(x)\in \ZZ[x]$ conditioned
on $p \nmid \Disc(f)$  to have  a given splitting type $\ssy$ exist and are given by  
the values $\nu_{n, p}^{\ast}(\cssy)$ (See Theorem \ref{th12}).  
For  a  fixed splitting type  $\ssy$   
 the values $\nu_{n, p}^{\ast}(\cssy)$  as functions of the prime $p$  are  interpolated by a rational function $R_{\mu}(z)\in \CC(z)$,
where we have $R_{\mu}(p)=\nu_{n, p}^{\ast}(\cssy)$ holding for each  $p$.
 This {\em rational function interpolation property} yields  a parametric family of (complex-valued) measures 
$\nu_{n, z}^{\ast}$  on  $S_n$
for $z \in \PP^1(\CC) \smallsetminus \{ 0, 1\}$,  termed {\em $z$-splitting measures}.
\item[(ii)]
The $z$-splitting measure is a positive probability measure
whenever $n$ is an integer greater than $1$  and $z= t$ is a real number greater than $n-1$.  
The uniform distribution on $S_n$ is the $z$-splitting measure for  $z= \infty \in \PP^1(\CC)$
(See Theorem \ref{th20a}). 
\item[(iii)]
There exist  infinitely many  $S_n$-number fields having prescribed
splitting types $(p_i, \ssy_i)$ at a given finite set of primes $S=\{ p_1, ..., p_r\}$, 
provided that all the splitting types have $\nu_{n, p_i}^{\ast}(\cssy_i)>0$.
The latter conditions are satisfied if and only if there exists an $S_n$-number field $K$ with a subring of algebraic integers
that is a monogenic order with discriminant
relatively prime to $\prod_i p_i$. The 
 existence of one such number field $K$
 certifies that the associated probability $\nu_{n, p_i}^{\ast}(\ssy_i)>0$
(See Theorem \ref{th13}).
\item[(iv)]
For each $n \ge 2$  there is a finite set of {\em exceptional pairs} $(p_i, \ssy_i)$ having 
$\nu_{n, p_i}^{\ast}(\cssy_i)=0$.
The exceptional primes  $p_i$ necessarily satisfy $2 \le p_i \le n-1$, and this
set is nonempty for $n \ge 3$.
The   exceptional pairs correspond to the condition that
all $S_n$-number fields having such a splitting type $(p_i, \ssy_i)$ have the prime $p_i$
as an {\em essential discriminant divisor} (a notion  defined in Section \ref{sec23}) (Theorem \ref{th25}). 
The phenomenon of essential discriminant divisors was first noted in 1878 by  Dedekind \cite{Dedekind:1878}.
\end{enumerate}

The $z$-splitting  measures $\nu_{n,z}^{\ast}$  seem of intrinsic interest, and arise in contexts not considered in this paper. First, 
the measure   $\nu_{n,k}^{\ast}$  may also have an interesting
 representation theoretic interpretation  for integer values $z=k$,  viewing   the measure   as 
specifying a rational character  of $S_n$. The  first author will show that this
is the case for $z=1$, where the measure is a signed measure supported on
the Springer regular elements of $S_n$ \cite{Lagarias:2014}. Second, for $z=p^k$
these measures  arise  in a fundamental example  
 in  the  theory of representation stability being developed by
  Church, Ellenberg and Farb \cite{CEF:2013a}.  \cite{CEF:2013b},
see  Section \ref{sec72}.


\subsection{Bhargava $S_n$-number field  splitting model }\label{sec12}

The  probability model for polynomial factorization $(\bmod \, p)$ 
studied in this paper  has strong parallels  with a  probability model 
developed by  Bhargava \cite{Bhargava:2007}
 for the splitting of primes in certain number fields $K/\QQ$ of degree $n$.

 Bhargava defines an   {\em $\,S_n$-number field} $\,\,K/\QQ\,\,$
 to be a number field with \\
 $[K: \QQ]=n$ whose Galois closure $L$  over $\QQ$ has Galois group
$S_n$.  Thus  $[L:\QQ] = n!$ while $[K:\QQ]=n$. An $S_n$-number  field $K$ is a non-Galois extension of $\QQ$
for $n \ge 3$.
Bhargava's probability  model  takes  as  its sample
space, with parameter $D$, the set of all $S_n$-number fields $K$ of discriminant 
$|D_K| \le D$ with the uniform distribution; his 
results and conjectures concern  limiting behavior of the splitting densities at 
a fixed prime $p$ as $D \to \infty$,
conditioned on the restriction that the field $K$ be unramified  at $(p)$, i.e. $p \nmid D_K$, the (absolute) field
discriminant of $K$. 
He formulates conjectures about these limiting distributions for splitting of a fixed prime $(p)$ and proves them  for $n\le 5$.
These conjectures  are  unproved for  $n \ge 6$. 

There is a close connection of $S_n$-number fields  with  $S_n$-polynomials, which
relates the two models. 
Any  primitive element  $\theta$ of an $S_n$-field  that is an algebraic integer has  $\theta$ being a root of
an $S_n$-polynomial. Conversely, adjoining a single root of an $S_n$-polynomial $f(x)$ always yields
a field $K= \QQ(\theta)$ that is an $S_n$-extension in Bhargava's sense. 
In the case that  $p \nmid \Disc(f)$, where $\Disc(f)$ is the polynomial discriminant, the splitting type
of the polynomial  $ f(x) \, (\bmod \, p)$ determines the splitting type of the prime ideal $(p)$ in  $K/\QQ$,
and also the Artin symbol $[ \frac{K_f/\QQ}{(p)}]$ (which is a conjugacy class in $S_n$).
The probability model of this paper can then be interpreted as
studying pairs $(K, \alpha)$
in which $K$ is an $S_n$-number field, given with an element $\alpha \in O_{K}$ such that
$K= \QQ(\alpha)$, with a finite sample space specified by size restrictions on the coefficients
that the (monic) minimal polynomial of $\alpha$ satisfies. Bhargava's model samples fields $K$ with one distribution, while
the model of this paper samples $(K, \alpha)$ with another distribution.

We discuss Bhargava's  model in detail in Section \ref{sec2a}.
 Our main observation is  that the limiting probability distributions of the two models do not
agree: the $p$-splitting measures depend on $p$, while Bhargava's
measures are the uniform measure on $S_n$,
which is independent of $p$.  We also observe that Bhargava's limit measure, the uniform measure on $S_n$, arises as  the $p \to \infty$ limit of the $z$-splitting measures.  
 In Section \ref{sec25} we present a detailed comparison 
of the structural features of the models, and identify 
differences. However  we do not
have  a satisfying conceptual explanation that accounts for the
  differences of the limiting 
probabilities in the two models, and leave finding one as
an open question..


\subsection{Plan of Paper}

Section \ref{sec2} states the main results.
Section \ref{sec2a} discusses Bhargava's number field splitting model and
compares its predicted probability distributions with the model of this paper.
 Section \ref{sec3}  derives basic properties of the splitting probabilities.
 Section \ref{sec4}  obtains  the limiting distributions
of splitting probabilities for polynomials with integer coefficients in a box. These splitting probabilities are essentially inherited
from the analogous splitting probabilities for random monic polynomials over
finite fields, see Section \ref{sec35}.
We also establish result (iii) above on  existence of infinitely many $S_n$-number fields having
given splitting types at a finite set of primes, avoiding exceptional pairs.
Section \ref{sec5}  extends the splitting results of this paper  to monic polynomials
with coefficients in rings of integers of a fixed number field,
choosing boxes based on a fixed $\ZZ$-basis of the ring of integers. The answer  involves the
splitting measures $\nu_{n, q}^{\ast}(\cssy)$ for $q= q=p^f$, $f \ge 1$.
This generalization is an application of results of S. D. Cohen \cite{Cohen:1981}.
Section \ref{sec6} discusses generalizations of the splitting problem to random matrix
ensembles,  as well as other appearances of $z$-splitting densities.  \medskip

{\bf Notation.}
Our  notation for partitions  differs from Macdonald    \cite{Macdonald:1995}.
We denote partitions of $n$ by  $\ssy= (\ssy_1, ..., \ssy_k)$,  
with $\ssy_1 \ge \ssy_2 \ge \cdots \ge \ssy_k$,  where Macdonald uses $\lambda$;
and  the multiplicity of part $i$ of $\ssy$ is  denoted $c_i(\ssy) :=|\{j: \ssy_j =  i\}|$,
where  Macdonald uses $m_i(\lambda)$. We sometimes write a partition  of $n$ in bracket notation as
$\ssy= \langle 1^{c_1}, 2^{c_2}, \cdots ,n^{c_n}\rangle$, with only  $c_i=c_i(\mu)>0$ included,
following Stanley \cite{Stanley:2012}.


\section{Results}\label{sec2}

\subsection{Splitting Measures } \label{sec21}

The results in this paper are expressible  in terms of 
a discrete family of probability distributions on the symmetric group $S_n$ indexed by $q=p^k$.
These distributions belong to  a  one-parameter family of 
complex-valued measures  on $S_n$,
 depending on a parameter  $z \in \CC \smallsetminus \{ 0, 1\}$
given below, which we call  {\em $z$-splitting measures.}
Restricting the parameter to real values $z=t \in \RR \smallsetminus \{ 0, 1\}$
we obtain signed measures of total mass $1$, and all the parameter values $t=q= p^k$
which are prime powers give nonnegative probability measures on $S_n$;
these measures originally arose in 
statistics involving  the factorization
of random square-free polynomials over $\FF_q[X]$, see Section \ref{sec35}.

\begin{defi} \label{de20a}
{\em For each degree $m \ge 1$  the $m$-th
{\em necklace polynomial} $\nmx$  by
$$
\nmx := \frac{1}{m} \sum_{d|m} \mu(d) X^{m/d}.
$$
where $\mu(d)$ is the M\"{o}bius function. }
\end{defi}

The necklace  polynomial takes integer values at  integers $n$,
its values at positive integers  have an enumerative interpretation that
justifies its name, given in  Section \ref{sec31}. 
These polynomials  arise in our context because for  $X= q=p^f$ a prime power, $M_m(q)$ counts the number of 
irreducible monic degree $m$ polynomials in  $\FF_q[X]$, where $\FF_q$ is the finite field
with $q$ elements, see Lemma \ref{le31}.

For a given element $g \in S_n$, denote its cycle structure (lengths of cycles)
 by
$\ssy= \ssy(g) =: \left(\ssy_1,\ \ssy_2,\ \ldots\ \ssy_k\right)$  with $\ssy_1 \ge \ssy_2 \ge ...\ge \ssy_k$.
 Here we regard  $\ssy$ as an {\em unordered  partition} of $n$, though for convenience we
 have listed its elements in decreasing order, and we denote it  $\ssy \vdash n$.
 The conjugacy classes on $S_n$ consist of all elements $g$ with a fixed cycle structure
 and we denote them $C_{\ssy}$. For a partition $\ssy \vdash n$ we  let
  $$
  c_i=  c_i(\ssy) := | \{ j: \ssy_j= i\}|
  $$
   count  its number of parts of size $i$, and we sometimes denote it 
  by  the bracket notation
   $\ssy  = \langle 1^{c_1}, 2^{c_2}, \cdots ,n^{c_n}\rangle,$
  with only $c_i>0$ included.


\begin{defi}\label{de21a}
 {\em The  {\em $z$-splitting measure} $\nu_{n,z}(g)$ for $g \in S_n$  is given by
 \begin{equation}\label{101}
 \nu_{n, z}^{\ast}(g) :=
  \frac{1}{n!} \cdot \frac{1}{z^{n-1}(z-1)} \prod_{i=1}^n i^{c_i} c_i! \binom{M_i( z)}{c_i(\ssy)},
  \end{equation}
 where for a complex number $w$ we interpret 
 $\binom{w}{k} := \frac{(w)_k}{ k!}= \frac{w(w-1) \cdots (w-k+1)}{k!}$.
 }
 \end{defi}

For each fixed $g \in S_n$ the quantity $ \nu_{n, z}^{\ast}(g) $
is a rational function of $z$, and is well-defined away from the polar set,
which is contained in $z=0, 1$.
  The splitting measure of an
  individual  element  $g$ depends only on its cycle type $\ssy=\ssy(g)$, 
  so  is constant on
conjugacy classes $C_{\ssy}$ of $S_n$. Using the well known formula
\begin{equation}\label{112a}
|C_{\ssy}| = n!  \prod_{i=1}^n \,  \frac{i^{-c_i(\ssy)}}{c_i(\ssy)!},
\end{equation}
for the size of conjugacy classes \cite[Proposition 1.3.2]{Stanley:2012}, we obtain
\begin{equation}\label{113a}
\nu_{n, z}^{\ast}(C_\ssy) := \sum_ {g \in C_{\ssy}} \nu_{n, z}^{\ast}(g)=
\frac{1}{z^{n-1}(z-1)} \prod_{i=1}^n \binom{M_i(z)}{c_i(\ssy)}.
\end{equation}

 Properties of these measures are studied in Section \ref{sec3}.
  The measures are defined by the right side of \eqref{113a}  as complex-valued measures
 for all $z$
  on the Riemann sphere, excluding $z=0$.
The definition implies that  they
 have total mass one, in the sense that
 $$
 \sum_{g \in S_n} \nu_{n, z}^{\ast} (g) = 1.
 $$
In this paper we  restrict to real values  $z=t$, in which case $ \nu_{n, t}^{\ast} (g)$
 in general  defines a signed measure on $S_n$.
  In Section \ref{sec34a} we prove results specifying positive   real $z$-values
 where the $z$-splitting measure is nonnegative.
 In particular we show nonnegativity of the measure holds for all 
 positive integers $t=m \ge 2$.
 
 \begin{thm}\label{th20a}
 Let $n \ge 2$. The $z$-splitting measures $\nu_{n, z}^{\ast}$ have
 the following properties, for positive real parameters $z=t>1$.
\begin{enumerate}
\item
For all real $t > n-1$, one has
$$
\nu_{n, t}^{\ast}(g) > 0~~\mbox{for all}~~ g \in S_n,
$$ 
For these parameter values  $\nu_{n, t}(\cdot)$ is  a probability measure  with full support
on $S_n$.
\item
For integers $k=2, 3, ... , n-1,$ one has
$$
\nu_{n, k}^{\ast}(g) \ge 0~~\mbox{for all}~~ g \in S_n,
$$ 
so that $\nu_{n, k}^{\ast}(\cdot)$ is  a probability measure on $S_n$. For these parameter values
this measure does not have full support on $S_n$.  It is zero on the conjugacy class of the identity element
 $C_{\langle 1^n\rangle}$.

\item
As $t \to \infty$ through positive real values,  one has
$$
\lim_{t \to \infty} \nu_{n, t}^{\ast}(g) = \frac{1}{n!}.
$$
\end{enumerate}
\end{thm}

 In Section \ref{sec34b} we prove a complementary result
 specifying negative  real  $z$-values
 where the $z$-splitting measure is nonnegative.
 In particular, nonnegativity holds  for all negative
 integers  $m \le -1$ (Theorem \ref{th20b}). This result is
 not used elsewhere in the paper.
 
 We also note that a later result (Theorem \ref{th25}) below refines case (1) of Theorem \ref{th20a} to characterize 
  for each $n$ all pairs $(p, \ssy)$ with 
$p$ a prime and   $\nu_{n, p}^{\ast}(C_{\ssy}) =0.$


\subsection{Prime Splitting Densities of $S_n$-Polynomials}\label{sec22}

  Let $f(x) \in \ZZ[x]$ be a monic polynomial. 
  Consider for a  prime $p$ the splitting of  such polynomials
$(\bmod~p)$, viewed in $\FF_p[X]$.

More generally for $q=p^f$,  any monic  $f(x) \in \FF_q[x]$ factors uniquely as 
$f(x) = \prod_{i=1}^kg_i(x)^{e_i}$, where the $e_i$ are positive integers and the 
$g_i(x)$ are distinct, monic, irreducible, and non-constant. 
We may define the \emph{splitting type}  of such a polynomial (following Bhargava \cite{Bhargava:2007}) to be the formal symbol 
$$
\ssy_q(f) := \left( \deg(g_1)^{e_1},\ \deg(g_2)^{e_2},\ \ldots ,\deg(g_k)^{e_k} \right)
$$
 where $k$ is the number of distinct irreducible factors of $f(x)$. 
 Here we order the degrees in decreasing order.
 We  let $\tns$ denote the set
of all possible formal symbols for degree $n$ polynomials,
which we call {\em splitting symbols}.
 Thus $T_3 =\{ (111), (21), (3), (1^2 1), (1^3)\}$.
 Using this definition, given any monic $f(x) \in \ZZ[x]$ and any prime $p$,
we may assign to it a splitting type 
$\ssy_p (f) \in \tns$. 

This paper mainly restricts to {\em square-free splitting types}, which are those having all $e_i=1$.
We define $\tnss \subset \tns$ to denote the set of such splitting types.
  Thus $T_3^{\ast}= \{ (111), (21), (3)\}.$
  Each  element $\ssy :=  \left(\ssy_1,\ \ssy_2,\ \ldots\ \ssy_k\right)  \in \tnss$ 
  with $\ssy_1 \ge \ssy_2 \ge ...\ge \ssy_k$  specifies a partition of
 $n$, to which there is 
 associated  a unique conjugacy class  $C_{\ssy} \subset S_n$.
The conjugacy class $C_{\ssy}$ is the set of all elements of $S_n$ whose 
cycle lengths are equal to the (unordered) numbers 
 $\ssy_1, \ldots, \ssy_k$.  In  this case we  will   refer also to $C_{\ssy}$ as a  {\em splitting type},
 and if $\ssy_p(f) \in \tnss$ then we will write $\mu_p(f) = C_{\ssy}$.

Given any positive integer $n$ and a positive number
  $B$ we let $\mathcal{F}_{n}(B)$ denote the collection of all degree $n$ monic
  polynomials  with integer coefficients, 
  $$f(x) = x^n + \sum_{j=0}^{n-1} c_j x^j \in \ZZ[x],$$
  having  coefficients bounded by $-B < c_j \le B,$ for $0 \le j \le n-1$.
  Then  let   $\mathcal{F}_{n,B,p}$
 be the subset of monic polynomials in $\mathcal{F}_{n}(B)$ having the following properties:
  \begin{enumerate}
  \item[(i)]
  The polynomial discriminant $(\Disc(f), p) =1.$
    \item [(ii)]
  All coefficients of $f(x)$ are contained in  $[-B+1, B]$. This implies that   the
  polynomial discriminant
  $|\Disc(f)| \le (4B)^{n(n-1)}.$
   \item[(iii)]
  $f(x)$ is irreducible over $\QQ$
  and  the  degree $n$ number field $K_f =\QQ(\theta_f)$ generated by one root  has normal closure
  with Galois group $S_n$.

 \end{enumerate}
 The allowed splitting types $(\bmod \, p)$ of polynomials in $\mathcal{F}_{n}(B;p)$
are constrained by the  requirement (1) on the discriminant  to belong to $\tnss$, i.e. to
 be square-free $(\bmod~p)$.
For this case we show: \medskip


\begin{thm}~\label{th12} {\rm (Limiting Splitting Densities)}
Let  $n \ge 2$ be given. Then:

(1) For each  prime $p$, there holds
\begin{equation}\label{111}
\lim_{ B \to \infty} \frac{ \#\{ f(x) \in \mathcal{F}_{n}(B; p)\}}
{\#\{ f(x) \in \mathcal{F}_{n}(B)\}} = 1 - \frac{1}{p}.
\end{equation}

(2) For  each (square-free) splitting type $\ssy \in \tnss$, there holds
\begin{equation}\label{112}
\lim_{ B \to \infty}
\frac{\#\{f(x) \in \mathcal{F}_{n}(B;p)\mid (p) \mbox{ has splitting type } ~~
C_{\ssy} \}}{\#\{ f(x) \in \mathcal{F}_{n}(B;p)\}}
 =  \nu_{n, p}^{*}(C_{\ssy}),
\end{equation}
where $\nu_{n, p}^{\ast}$ is the 
splitting measure for $n$ with parameter $t=p$, i.e.
\begin{equation}\label{113}
\nu_{n,p}^{*}(C_\ssy)= \frac{1}{p^{n-1} (p-1)} \prod_{i=1}^n {\nip \choose c_i(\ssy)}.
\end{equation}
\end{thm}


This result is proved in Section \ref{sec4}; it is   a special case $r=1$ of Theorem \ref{th42}
which applies more generally to finite sets $\sS=\{p_1, p_2, ..., p_r\}$ of primes.
 In Section \ref{sec5} we give a further generalization of the
result to algebraic number fields.


\subsection{Existence of $S_n$-Number Fields with Prescribed Prime Splitting}\label{sec23}

We also  show  there are infinitely many $S_n$-number fields with prescribed prime splitting
at any finite set $\sS$ of primes, of those types allowed by the splitting measures. 
The splitting measures impose some extra constraints associated to the existence of monogenic
orders in the $S_n$-number fields having discriminants relatively prime to given elements. \medskip

\begin{thm}~\label{th13}
Let $n \ge 2$ be given, let $\sS = \{p_1, ..., p_r\}$ denote a finite set of (distinct) primes, and let
$\SSY =\{ \ssy_1, ..., \ssy_r\}$  be a prescribed set of (not necessarily distinct) splitting symbols
for these primes.  Then the  following conditions are equivalent.

\begin{enumerate}
\item[(1)]
The positive measure condition
$$
\nu_{n, p_i}^{\ast}(C_{\ssy_i}) >0  \,\, \mbox{for}\, \, 1 \le i \le r
$$
holds.

\item[(2)]  There exists an
$S_n$-number field $K$  having the following two properties:

\begin{enumerate}
\item[(P1)] The field $K$ contains a monogenic order $O= \ZZ[1, \theta, ...\theta^{n-1}]$
whose discriminant is relatively prime to $p_1p_2 \cdots p_r.$
\item[(P2)]
The  Galois closure $K^{spl}$ of $K/\QQ$ 
  is unramified at all prime ideals  above those in $\sS$ and  the primes 
 in $\sS$ have  prescribed Artin symbols
$$
\Big[ \frac{K^{spl}/\QQ}{(p_i)}\Big] = C_{\ssy_i}\,, ~~~~~~1 \le i \le r.
$$
\end{enumerate}
\item[(3)] There exist infinitely many $S_n$-number fields $K$ having
properties (P1) and (P2).
\end{enumerate}
\end{thm}

The condition (1) automatically holds when all $p_i \ge n$,
because the probability measure $\nu_{p_i, n}^{\ast}$ then
has full support on the group $S_n$. However for  primes $2 \le p<n$  there are  restrictions on
the allowed splitting behavior. This restriction has to do with the non-existence of
monogenic maximal orders satisfying (P1) for $S_n$-number fields having  specific prime factorization
at small prime ideals. The polynomials $f(x)$ generating such fields have
{\em essential discriminant divisors}\footnote{
Related concepts  include  
the {\em inessential discriminant
divisor} $I(K)$ of a field $K$  (Tormhein \cite{Tornheim:1955}),
also called the {\em non-essential discriminant divisor} of $K$ (Sliwa \cite{Sliwa:1982}).
Here $I(K) = \gcd_{\theta \in O_K} i(\theta),$  where
$i(\theta) := [O_K: \ZZ[1, \theta, ..., \theta^{n-1}]].$ The {\em essential discriminant divisors} are
exactly the prime divisors of  $I(K)$.}, 
as defined in  Cohn \cite[Defn. 9.55, Lemma 10.44c]{Cohn:1978} and Cohen \cite[p. 197]{Cohen:1993}. 
A famous example due to Dedekind \cite{Dedekind:1878} (see \cite[Exercise 9.4; Lemma 10.44c]{Cohn:1978})
 is  an  $S_3$-number field $K$ for which the prime ideal $(2)$ splits completely in $K$;  all monogenic
orders then have an even index, and correspondingly $\nu_{3,2}^{\ast}([1^3]) = 0$.
However it is  known that infinitely  many $S_3$-number fields $K$ 
 exist  in which the ideal $(2)$ splits completely in the maximal order.
This result follows from  results of Bhargava for $n=3$ discussed in Section \ref{sec24}.
 Such fields are not covered by Theorem \ref{th13}.

Theorem \ref{th13} allows us to characterize
the splitting measures for prime values  $t= p \ge 2$ having probability $0$
in terms of field-theoretic data. 


\begin{thm}\label{th25}
 For $p$ a prime, and a splitting type $\ssy \vdash n$,  
for fixed $n \ge 2$, the following three conditions are equivalent.
\begin{enumerate}
\item[(C1)] The splitting measure at $t=p$ has
$$
\nu_{n, p}^{\ast} ( C_\ssy) = 0.
$$
\item[(C2)]
There are no degree  $n$ monic  polynomials $f(x) \in \ZZ[x]$ 
 with $f(x) \,\pmod  p$ having a square-free factorization
of splitting type $C_{\ssy}$.

\item[(C3)] All $S_n$-number fields $K$  in which  $(p)$ is unramified and 
has splitting type $\ssy$ necessarily have $p$ as an essential discriminant divisor.
\end{enumerate}
\end{thm}

 This  result is proved in Section \ref{sec44}. 
 The condition (C1) is vacuous for $n=1, 2$.
 This theorem  provides
 the easy-to-check  criterion (C1) for an $S_n$-number field $K$ to have $(p)$ as an
 essential discriminant divisor, via the splitting type $\ssy$ of $(p)$ in $K$.   
 Condition (C2) is a statement about  all $f(x) \in \ZZ[x]$; it does not require $f(x)$ to be an $S_n$-polynomial or  to be
   irreducible over $\QQ$.
 Our  proof does not show the existence of even a single
 field satisfying (C3) for any given pair  $(p, \ssy)$ satisfying (C1). 
 Conjecture \ref{conj51} of Bhargava below would imply that infinitely many such fields exist,
 and this conjecture is known to be true for $n \le 5$.  
 
In  Section \ref{sec5} we establish generalizations of  Theorem \ref{th13} and Theorem \ref{th25} above 
in which  the base field $\QQ$ is
replaced by an algebraic number field $K$.
These generalizations  are stated as Theorems \ref{th52}
and \ref{th53}, respectively. These generalizations are a more complicated to state, and
their proofs are straightforward, using results of S. D. Cohen \cite{Cohen:1981}.


\section{Bhargava number field splitting model}\label{sec2a}

 Recall from Section \ref{sec12} that Bhargava defines an   {\em $S_n$-number field} $K/\QQ$
 to be a number field with $[K: \QQ]=n$ whose Galois closure $L$  over $\QQ$ has Galois group
$S_n$.  
Bhargava's  number field splitting  model  has  sample
space ithe set of all $S_n$-number fields $K$ of discriminant 
$|D_K| \le D$ with the uniform distribution, and his 
results and conjectures concern  the limiting behavior of splitting densities at 
a fixed prime $p$ as $D \to \infty$,
conditioned on the property that the  field $K$ be unramified  at $(p)$, i.e. $p \nmid D_K$, the (absolute) field
discriminant of $K$.


\subsection{Bhargava's conjectures for prime splitting in $S_n$-number fields}\label{sec24}

In 2007 Bhargava \cite{Bhargava:2007}  formulated  conjectures about the splitting of primes 
averaged over $S_n$-number fields 
ordered by the size of their field discriminants. 
 Bhargava developed his conjectures based on 
 the following principle \cite[p. 10]{Bhargava:2007}:
 
 \begin{quote}
 The expected (weighted) number of global $S_n$-number fields of discriminant
 $D$ is simply the product of the (weighted) number of local extensions of $\QQ_{\nu}$ that
 are discriminant-compatible with $D$, where $\nu$ ranges over all places of $\QQ$,
 (finite and infinite).
  \end{quote}

In this statement a {\em  local extension of $\QQ_v$}  means a degree $n$ \'{e}tale algebra $E$ over $\QQ_v$
(not necessarily a field) and {\em discriminant-compatible} means that the valuation of the discriminant
of $E$ matches that of $D$ and that, in the archimedean case,   the signs of the
discriminants match. 
We state two of his  conjectures below in order to later compare them with our results.

Firstly, given any positive integer $n$ and a positive number
  $B$ we let $\mathcal{G}_{n}(B)$ denote the collection of
  $S_n$-number fields $K$ that have discriminants $|D_K| \le B$.
Secondly, given any positive integer $n$, prime $p$ and positive number
  $B$ we let $\mathcal{G}_{n}(B;p)$ denote the collection of all degree $n$ number fields 
  $K$ such that: 
  \begin{enumerate}
  \item[(i)]
  The ideal $(p)$ is unramified in $K$;
   \item[(ii)]
 The  field discriminant $|D_K| \le B$;
  \item[(iii)]
  The degree $n$ field  $K$ over $\QQ$ 
  has normal closure having Galois group $S_n$.
  \end{enumerate}
 
 The first conjecture  of Bhargava
 concerns which fraction of $S_n$-number fields have 
 field discriminant $K$ relatively prime to $p$ \cite[Conj.\ 1.4]{Bhargava:2007}.
 
\begin{conjecture}[Bhargava]
\label{conj24}
Fix a prime $p$ and a  positive integer $n$.
Then 
\begin{equation}\label{141}
\lim_{B\rightarrow \infty}
\frac{\#\{K \in  \mathcal{G}_{n}(B;p) \}}{\#\{ K \in \mathcal{G}_{n}(B)\}} = 
1- \pram_{n}(p).
\end{equation}
where $\pram_{n}( p)$ is the ``probability of ramification," given by
\begin{equation}\label{142}
\pram_{n}( p) := \frac{\sum_{k=1}^{n-1} q(k, n-k) p^{n-1-k}}{\sum_{k=0}^{n-1} q(k, n-k) p^{n-1-k}},
\end{equation}
in which $q(k,n)$ denotes the number of partitions of $k$ into at most $n$ parts.
\end{conjecture}

By convention we set $q(0, n)=1$ for $n \ge 1.$
This distribution  $\rho_{n}(p)$ depends 
on both $n$ and $p$ and is a rational function of $p$.
For fixed $n$,
 $\rho_n(p) =  \frac{1}{p}+O(\frac{1}{p^2}) $  as $p \to \infty$.

Bhargava proves Conjecture \ref{conj24}  
 for $n \le 5$.  For  these cases, the 
probabilities are  $\pram_{1}(p)=0$ and
$ \pram_{2}(p)= \frac{1}{p+1}, \, \pram_3(p) = \frac{p+1}{p^2 + p +1},
\pram_{4}(p) = \frac{p^2+ 2p +1}{p^3 + p^2 + 2p +1},$
and
$\pram_{5}(p)= \frac{p^3+2p^2+2p+1}{p^4 + p^3 + 2p^2 + 2p +1},$ 
respectively.  
 In another conjecture, Bhargava \cite[Conjecture 5.2]{Bhargava:2007}
 further relates these probabilities  to the distribution of splitting types in $T_n$ having repeated factors.


Bhargava's second conjecture   
about prime splitting in $S_n$-number fields  is as follows
\cite[Conj.\ 1.3]{Bhargava:2007}.

\begin{conjecture}[Bhargava]
\label{conj51}
Fix a prime $p$, a positive integer $n$, and $\ssy \in \tnss$. Then 
\begin{equation}\label{151}
\lim_{B\rightarrow \infty}
\frac{\#\{K \in \mathcal{G}_{n}(B,p)\}\mid p\mbox{ has Artin symbol in} ~~
C_{\ssy}\}}{\#\{ K \in \mathcal{G}_{n}(B;p)\}} = \nssy,
\end{equation}
where $\nu_n(\cdot) $ denotes the Chebotarev density distribution on
conjugacy classes of  $S_n$, which is
$$
\nu_n(C_{\ssy}) := \frac{|C_{\ssy}|}{|S_n|} 
$$
\end{conjecture}

Conjecture \ref{conj51} predicts that the limiting density exists
and agrees with that predicted by the  Chebotarev density 
theorem for conjugacy classes (see \cite{Lagarias-O:1977}, 
\cite[Chap. 7, \S3]{Narkiewicz:1990}, \cite[ Theorem 13.4]{Neukirch:1999}); 
this measure corresponds to the uniform distribution on $S_n$.
This limiting distribution depends on $n$ but is independent of $p$.
It is proved for $n \le 5$. The case $n=3$ is deducible from 
results of Davenport and Heilbronn \cite{Davenport:1971}, see
also Cohen et al \cite{CDDO:2000}. 
Bhargava proved the result for   $n=4$ and $n=5$ using his
earlier results  for discriminant density in quartic and quintic fields
\cite{Bhargava:2005, Bhargava:2010}.

For a general viewpoint on Bhargava's conjectures, see
Venkatesh and Ellenberg \cite[Section 2.3]{Venkatesh-E:2010}.
Bhargava's conjectures on  local mass formulas, were reinterpreted
in connection with Galois representations 
in Kedlaya \cite{Kedlaya:2007} and 
further cases were considered by Wood \cite{Wood:2008, Wood:2010}.


\subsection{Random polynomial model versus  random number field model}\label{sec25}

We compare the distributions  for prime splitting in
$S_n$ number fields in the random polynomial model 
against those of the random number field model
given in   Bhargava's conjectures.
These splitting distributions differ.

We summarize the comparison in Table \ref{t:comparison}. A main feature is  that 
for each $n \ge 1$ the densities random monic  polynomial model in the $p \to \infty$ limit approaches
the uniform  density distribution conjectured in Bharagava's model. 

\begin{table}[!hbtp]
\begin{tabular}{ |p{1.3 in}|p{1.7 in}|p{1.6 in}|}
\hline
Probability model& Random $S_n$-Polynomial  Model   & Random $S_n$-Number Field Model  (Bhargava) \\ \hline
\hline
Sample space & Degree $n$, monic polynomials with integer coefficients $|c_i| \le B$,
generating an $S_n$-number field& $S_n$-number fields  $K$ with  field discriminant $|D_K|$ bounded by $D$ \\ \hline
Limit procedure &  Box size $B \to \infty$ & Discriminant $D \to \infty$\\ \hline
 Ramification probability at $(p)$ & $\text{Prob}[ \,p \,\,\text{divides} \, Disc(f)]$ equals $ \frac{1}{p}$, which is independent of $n$&
 $\text{Prob}[ \,p \, \text{divides} \,  Disc(K)]$ is a quantity $\theta_n(p)$  which depends on both $n$ and $p$ (Conjecture \ref{conj24}) \\ \hline
 Limiting distribution on $S_n$ of
 splitting types & $p$-splitting distribution  $\nu_{n, p}^{\ast} (C_{\ssy})$ on conjugacy classes, whose probabilities  depend on
both $n$ and $p$
 & Chebotarev  distribution 
 $\nu_n(C_{\ssy})= \frac{|C_{\ssy}|}{n!}\,\,$,  which is independent of $p$ (Conjecture \ref{conj51} ) \\ \hline
   Limit $p \to \infty$ of
 ramification probability& 0 & 0 \\ \hline
Limit $p \to \infty$ of distribution densities & Uniform distribution $\nu_{n, \infty}^{\ast}= \nu_n$ on elements of $S_n$  & 
Uniform distribution $\nu_{n, \infty}^{\ast}= \nu_n$ on elements of $S_n$\  \\ \hline
\end{tabular}
\bigskip
\caption{Comparison of  polynomial splitting  model  and random $S_n$-number field model probabilities. (Conjectures
\ref{conj24} and \ref{conj51} are theorems for  $n \le 5$.)\label{t:comparison}}
\end{table}

Both model predictions assign
a weighted contribution of  $S_n$-number fields $K$ having  discriminant
prime to $p$, which  depend on the parameter
$B$ (resp. $D$), and consider the limiting distribution as the corresponding parameter grows.
In each model the splitting density is a conditional probability based on 
conditioning against an ``unramifiedness" condition. 
There is  a  difference of scale in the cutoffs  in the $B$ and $D$ parameters
between the two models, in that polynomial discriminants
$D_f$ grow proportionally to  $B^n$.
 However the  limit as the parameters go to infinity, this scale 
differences play no role.

The main differences in the  predicted probabilities in  the models are
the following.

\begin{enumerate}
\item[(1)]
In Bhargava's conjectures the probability of ramification $\rho_n(p)$ depends on both
the prime $p$ and the degree $n$. One has
$$
\theta_n(p) := 1- \rho_{n}(p) = \frac{1}{1 + \sum_{k=1}^{n-1} q(k, n-k) p^{-k}},
$$
This formula implies that for fixed $p$ and variable $n$ the function  $\rho_n(p)$ increases  to the limit
$\rho_{\infty}(p) := 1- \frac{1}{P(1/p)}$
where 
$$
P(x) := \sum_{n=0}^{\infty} p_n x^n= \prod_{n=1}^{\infty} ( \frac{1}{1-x^n}).
$$
In contrast, in the random polynomial model  the probability of
ramification $ \frac{1}{p}$  is independent of $n$, according to Theorem \ref{th12} (1).
The formula above implies that for fixed $n$ one has $\rho_{n}(p) = \frac{1}{p} +O( \frac{1}{p^2}) $ as $p \to \infty$,
so both ramification probabilities go to $0$ as $p \to \infty$ at the same rate.

\item[(2)]
In Bhargava's conjectures the splitting probabilities are independent
of both $p$ and $n$. In contrast, in the random polynomial model  the probabilities  $\nu_{n, p}^{*}(C_{\ssy})$
depend on both $n$ and $p$.
\end{enumerate}

What features of the models account for the differing answers in the two models?
The models themselves have  structural differences.

\begin{enumerate}
\item[(D1)]
 The (irreducible) polynomial $f$ is associated algebraically
 not with the ring of integers $O_F$ of the
 field $K= \QQ(\theta)$ generated by a root $\theta$ of $f$, but with the particular monogenic order
 $O_f = \ZZ[1, \theta, \theta^2, \cdots, \theta^{n-1}]$. In particular 
 discriminant  $\Disc(f) = D_K c^2$, where $c= [O_K: O_f]$ is the index of $O_f$ inside
 $f$.  In particular $\Disc(f)$
 may be  divisible by primes which do not divide 
$D_K$, so the ``unramified" conditions of the two probability models differ.  
For some $S_n$-number fields $K$ the ring
of integers $O_K$ is not monogenic. The number of monogenic orders of a given
index in the maximal order $O_K$ (isomorphism up to an additive shift of a variable) is known  to 
depend on the index within a given field $K$, cf.  Evertse \cite{Evertse:2011}.
\item[(D2)]
Many different polynomials in $\mathcal{F}_{n}(B;p)$ generate the same
$S_n$-number field $K= K_f$. Thus each field  $K$ that occurs is weighted by the
number of polynomials in the box that generate it (and which satisfy
the discriminant co-primeness condition). The weights depend in
a complicated way on $K$ and $B$ and change as $B \to \infty$.
\end{enumerate}

The  difference  (D1)  of the ramification conditions in the two models
presumably accounts for much of the mismatch.
 The $S_n$-number fields detected by the random 
polynomial model are always
unramified in the field sense, but the random monic polynomial
models do not detect some $S_n$-number fields not
ramified at $(p)$.
We should really replace the $p$-part of $\Disc (f)$ with the $p$-part of
$D_K$, with $K = \QQ(\theta)$, which involves studying the $p$-adic coefficients of $f(x)$.
The model of Bhargava is based on a mass formula counting $p$-adic \'{e}tale extensions
with weights, and the weights matter.
However from  the viewpoint of the difference (D2) it is not immediately clear 
that such weighted sums  will  conspire to produce the
nice limiting values given  in Theorem \ref{th12}. To understand  difference (D2) better
 it might be interesting to study an auxiliary question: 
for each pair of $S_n$-number fields $K_1, K_2$ what is the behavior as $B \to \infty$ of the ratio
of the number of $f(x)$ in the box of size $B$ that generate the field $K_1$
(resp. $K_2$) and satisfy $p \nmid \Disc (f(x))$. Does this quantity have a
limiting value and if so, how does it depend on $K_1$ and $K_2$?

We conclude that there  are observable structural differences between the two models. 
We do not  currently have  a conceptual explanation how these 
structural differences account for and quantitatively explain
the  differences in the limiting densities of the two models. 


\section{Splitting Measures}\label{sec3}

In this section 
we define and study 
the one-parameter  family 
 of  splitting measures $\nu_{n,z}^{\star}$ ($z \in \CC$)
on the symmetric group $S_n$, for each $n$. 
 We relate this measure at $z=q= p^k$ to finite field factorization of degree $n$ monic
polynomials over $\FF_q$.


\subsection{ Necklace polynomials}\label{sec31}

The number of  monic irreducible polynomials
of degree $m$ over finite fields $\FF_q$ for $q=p^f$ are well known to 
be interpolatable by  universal
polynomial $M_m(X)$ evaluated at value $X= q$.
Recall that for   $m \ge 1$ the {\it necklace polynomial} of degree $m$ is
$\nmx \in \QQ[X]$  by 
\beql{200}
\nmx := \frac{1}{m} \left(\sum_{d | m} \mu(\frac{m}{d}) X^d \right) 
= \frac{1}{m} \left(\sum_{d | m} \mu(d) X^{\frac{m}{d}} \right),
\eeq
where $\mu(d)$ is the M\"{o}bius function. 
For $m=0$ we set $M_0(X)=1$.
We note  that $M_1(X)  = X$  and $M_2(X) =\frac{1}{2}X(X-1)$.
Clearly  $\nmx \in \frac{1}{m} \ZZ[X],$ for $m \ge 1$. 
The name ``necklace polynomial"  was proposed by   Metropolis and Rota \cite{Metropolis:1983},
because the  value  $M_m(k)$ for positive integer $k$ has a combinatorial interpretation
as counting  the number
of necklaces of $m$ distinct colored beads formed using $k$ colors which 
have the property of being  {\em primitive} in
the sense that  their cyclic rotations are distinct (Moreau \cite{Moreau:1872}).
In 1937 Witt \cite[Satz 3]{Witt:1937} showed that $M_m(k)$  counts the number of basic commutators of degree $m$ 
 in the free Lie algebra on  $k$ generators. 
See the discussion  in Hazewinkel \cite[Sect. 17]{Hazewinkel:2009}.

For later use we  give some basic properties of $\nmx$.
 
%
%
\begin{lemma} \label{le31}
(1) Let $q= p^k$ be a prime power and let $N_{m}^{irred}(\FF_q)$ count the
number of irreducible
monic polynomials in $\FF_q[X]$ of degree  $m$.
Then 
$$
M_m(q) = N_m^{irred}(\FF_q).
$$


 (2) The polynomial 
$\nmx \in \QQ[X]$ is an integer-valued polynomial, i.e.
one has $M_m(k)  \in \ZZ$ for all $k \in \ZZ$.
\end{lemma}

\begin{proof}
(1)  The well known formula
$N_m^{irred}(\FF_q)= M_m(q)$ 
was  found by  Gauss\footnote{Gauss found this formula  on August 25, 1797,
according to his {\em Tagebuch}, see  Frei \cite{Frei:2007}.}
 in the unpublished Section 8 of {\em Disquisitiones Arithmeticae,}
Articles 342 to 347, see  Gauss \cite[pp. 212--240]{Gauss:1876}, cf. Maser \cite{Maser:1889}.
A proof is given in  Rosen \cite[p. 13]{Rosen:2002}.

 (2)  To  verify that a polynomial in $\frac{1}{m}\ZZ[X]$ is integer-valued,
it suffices to  check the integrality property holds at $m$ consecutive integer values of $X$.
The integrality property at positive integers follows  from the counting interpretation of
 the values $M_m(j)$ of Moreau \cite{Moreau:1872}, 
 see also   \cite{Metropolis:1983}.
 \end{proof}

We next obtain  bounds on the size of $\nmx$ which will be used
in Sections \ref{sec34a} and \ref{sec34b} to establish 
non-negativity properties of the $z$-splitting distributions for certain
parameter ranges.

%
%
\begin{lemma} \label{le32}
(1) The necklace polynomial 
$\nmx$ has  $M_m(0)=0$  for $m \ge 1$. In addition
$$
M_m(1) = \left\{ 
\begin{array}{cl}
1 & ~~\mbox{for}~~m=1,\\
~&~\\
0 & ~~\mbox{for}~~m\ge 2.
\end{array}
\right.
$$
One has  $(X-1)^2 \nmid M_m(X)$ for all $m \ge 2$.

(2) One has
 $$
 M_m(t) >0, ~~~\mbox{for all real}~~ t \ge 2.
 $$
In addition, for  $1\le j \le m $  there holds for real $t>m-1$, 
\begin{equation}\label{223b}
M_j(t) > \left\lfloor \frac{m}{j}\right\rfloor -1.
\end{equation}

(3) For $m \ge 1$  one has
 \begin{equation}\label{minus-ineq}
(-1)^m M_m(-t) >0, ~~~\mbox{for all real}~~ t \ge 2.
 \end{equation}
 In addition,  for each $m \ge 2$ and  $t>0 $ with $t(t+1) > m-2$, there holds
for $1 \le j \le m/2$, 
\begin{equation}\label{223c}
M_{2j}(t) > \left\lfloor \frac{m}{2j}\right\rfloor -1.
\end{equation}
\end{lemma}

\begin{proof}
(1) We have $M_m(0)=0$ since it has no constant term for $m \ge 1$.
For $m \ge 1$ we have $M_m(1)  = \sum_{d|m} \mu(d),$ which
yields  $M_m(1) =0$ for $m \ge 2$. Thus $X(X-1) | M_{m}(X)$ for $m \ge 2$.
 The relation  $(X-1)^2$ does not divide  $\nmx$  follows from
$$
M^{'}_m(X)|_{X=1} = \frac{1}{m} \Big(\sum_{d |m} \mu(d) \frac{m}{d} \Big)=  \prod_{p|m}\Big(1- \frac{1}{p}\Big)> 0.
$$

(2) For $m \ge 2$ and real $t\ge 2$,  one has
\begin{equation}\label{eq33}
m \,M_{m}(t)  \ge t^m - \left(\sum_{j=1}^{\left\lfloor m/2\right\rfloor} t^j \right) = t^m - 
\left(\frac{ t^{\frac{m}{2} + 1} - t}{t-1} \right)\ge t^m - t^{\frac{m}{2}+1} + t> 0.
\end{equation}
For the second part, suppose $m \ge 2$ and $1 \le j \le m$. We have for $j=1$ and $t > m-1$
that
$M_1(t) = t > m-1.$
For $j=2$ and $t > m-1$ we  have
$$
M_2(t) =  \frac{1}{2}t(t-1) \ge  \left\lfloor \frac{m}{2}\right\rfloor -1,
$$
the last inequality being  immediate for $m=2$ and easy for $m \ge 3.$
Finally, for $3 \le j \le m$, and $t > m-1$, we have $t^j - t^{\frac{j}{2}+1} \ge 1,$ whence 
by \eqref{eq33}, 
$$
jM_j(t)\ge 1+t > m
$$
which gives (\ref{223b}) in this case.

(3)  To establish $M_m(-t) >0$ for $t>2$, note that for  $m=1$ one has $-M_1(-t) = t>0$ for $t>0$.
For $m \ge 2$ we have for $t>2$ that
\begin{equation}\label{eq33b}
m \,M_m(-t) \ge t^m - (\sum_{j=1}^{\left\lfloor m/2\right\rfloor} t^j ) 
\ge t^m - t^{\frac{m}{2}} + t
 > 0.
\end{equation}
For the second part, it suffices to show for $1 \le j \le m/2$ that
\[
(2j) M_{2j}(-t) > m -2j  \quad \Big(  \ge 2j( \left\lfloor \frac{m}{2j} \right\rfloor -1\Big).
\]
For $2j=2$ we have by hypothesis
\[
2M_{2}(-t) = t(t+1) > m-2.
\]
For $2j\ge 4$ and  $m \ge 6$, the condition  $t(t+1) > m- 2$ implies $t>2$.
Then  \eqref{eq33b} applies and we obtain.
\[
2j M_{2j}(-t) \ge t^{2j} - t^j + t \ge t(t+1) >  m -2 > m- 2j.
\]
as required. The remaining case is $m=4$ and $2j=4$,
where $m-2j=0$, the condition
$t (t+1)>2$ implies $t>1$, whence
\[
4M_{4} (-t) = t^4 - t^2 > 0,
\]
as required.

\end{proof}

\subsection{Cycle polynomials}\label{sec32a}

To any partition $\ssy \vdash n$ we associate the {\em cycle polynomial}
\beql{231}
N_{\ssy}( X) := \prod_{i=1}^n \Big( {{M_i(X)}\atop{c_i(\ssy)}} \Big)
\eeq
Here $N_{\ssy}(X) \in \QQ[X]$ is a polynomial of degree $n$ 
(since $\sum_{i=1}^n i c_i= n$).

 The values  $N_{\ssy}(X)$ 
for prime powers $X=q=p^f$  count the number of square-free
polynomial factorizations of type $\ssy$
 in $\FF_q[X]$, as shown in 
 Section \ref{sec35}.

%
%
\begin{lemma} \label{le34} {\em (Properties of Cycle Polynomials)}
Let $n \ge 2$. For any  partition $\ssy \vdash n$  
the cycle polynomial $N_{\ssy}(X)$ has the following properties.

(1) The polynomial 
$N_{\ssy}(X) \in \frac{1}{n!}\ZZ[X]$  is integer-valued.

(2) The polynomial $N_{\ssy}(X)$ has lead
term
$$
\left(\prod_{i=1}^n \frac{1}{ i^{ c_i(\ssy)} c_i(\ssy)!} \right) X^n = \frac{ |C_{\ssy}|}{n!} X^n.
$$

(3) The polynomial $N_{\ssy}(X)$ is divisible  by $X^{m}$,
where $m \ge 1$
counts  the number
of distinct cycle lengths appearing in $\ssy$.

(4) There holds
\begin{equation}\label{eq351}
\sum_{\ssy \vdash n} N_{\ssy}(X) = X^{n-1}(X-1).
\end{equation}
\end{lemma}

\begin{proof}
(1) The definition \eqref{231} implies that $N_{\ssy}(X) \in \frac{1}{d(\ssy)}\ZZ[X]$
with 
$$
d(\ssy) = \prod_{i=1}^n  i^{c_i(\ssy)} c_i(\ssy)!
$$
By comparison with equation (\ref{112a}) we  have $d(\ssy)= \frac{n!}{|C_{\ssy}|}$, which
shows that $d(\ssy)$ divides $n!$, with equality when $\ssy= \langle 1^n\rangle$.
The integrality of $N_{\ssy}(k)$  for $k \in \ZZ$ follows from the definition using the integrality
of all $M_i(k)$ (Lemma \ref{le31}(2)).

(2) The property follows by direct calculation of the top degree term in  (\ref{231}).

(3) The divisibility property is immediate from the definition (\ref{231})
since $X$ divides  $\Big( {{M_i(X)}\atop{c_i(\ssy)}} \Big)$ whenever $c_i(\ssy) >0$.

(4) Both sides of the identity (\ref{eq351}) 
are polynomials of degree $n$, so it suffices to
verify that the identity holds at $n+1$ distinct values of $X$. To this end, we make use of a combinatorial interpretation of $N_{\ssy}(X)$ for $X=q= p^k$
a prime power, given in
Proposition \ref{pr33} below. 
The sum on the left  evaluated at $X=p^k$ counts all possible  degree $n$ monic polynomials  over $\FF_q[X]$
for $q= p^k$ that have a square-free factorization, i.e. nonvanishing discriminant over $\FF_q$.
The resulting polynomial
$F(X)$ satisfies $F(q) = q^n - q^{n-1}$, according to Proposition \ref{pr33} (1), verifying
the identity at $X=q$.
\end{proof}


\subsection{ Splitting Measures}\label{sec34}

For each $n \ge 2$ we define the parametric family of (necklace) {\em splitting  measures} 
$\nu_{n, z}^{\ast}$ on the symmetric group $S_n$,
with family parameter $z \in \CC$,  by means of  their  values on conjugacy classes
\begin{equation}\label{spl-def}
\nu_{n, z}^{\ast}(C_\ssy) := \frac{1}{z^{n-1} (z-1)} N_{\ssy}(z) 
= \frac{1}{z^{n-1} (z-1)} \prod_{i=1}^n \binom{M_i(z)}{c_i(\ssy)}.
\end{equation}
For any element $g \in C_{\ssy}$ we set 
\begin{equation}\label{g-def}
\nu_{n, z}^{\ast}(g) := \frac{1}{|C_{\ssy}|} \nu_{n, z}^{\ast}(C_\ssy).
\end{equation}
The latter formula coincides with the definition  (\ref{101}) for $\nu_{n, z}^{\ast}(g)$.
Since this formula  is a rational function of $z$ for each $\ssy$,   with possible poles only at $z=0,1$,
this defines a complex-valued
function on $S_n$ constant on conjugacy classes, for all $z$ on the Riemann sphere
$\hat{\CC}:= \CC \cup \{ \infty\}$ except possibly at $z=0, 1$. The measure 
at $z= \infty$ is the uniform measure, $\nu_{n, \infty}(g) =\frac{1}{n!}$, a result that 
follows from Lemma \ref{le32}(ii)--see also the proof of Theorem \ref{th20a} (3) below. 
The measure  at $z=1$  also turns out to be well-defined
but  is now a signed measure. It is studied by the first author in \cite{Lagarias:2014}. 

We next show that these  measures have total (complex-valued) mass one.

\begin{prop}\label{pr35}
For $n \ge 1$,  for all $z \in \hat{\CC} \smallsetminus \{0\}$ and denoting
conjugacy classes in $S_n$ by $C_{\ssy}$  with $\ssy \vdash n$, 
$$
\sum_{\ssy \vdash n }  \nu_{n, z}^{\ast}(C_{\ssy})=1.
$$
Equivalently, for all $g \in S_n$,
$$
\sum_{g \in S_n} \nu_{n, z}^{\ast}(g)  = 1.
$$
\end{prop}

\begin{proof}
For $z \in \hat{\CC} \smallsetminus \{0, 1, \infty\}$
the  lemma follows from the normalization identity \eqref{eq351} for the cycle polynomials.
It extends by analytic continuation in $z$ to the  values $z= 1, \infty$.
\end{proof}


\subsection{ Splitting measures and finite field factorizations}\label{sec35}

A main  rationale for the study of $z$-splitting measures is that
when  $z=q= p^k$ is a prime power these  measures occur 
in the statistics of factorization of monic polynomials of degree $n$ in $\FF_q[X]$, 
drawn from a uniform distribution, conditioned on being square-free.

Recall that a  monic  polynomial $f(x) \in \FF_q[x]$ factors uniquely as 
$f(x) = \prod_{i=1}^kg_i(x)^{e_i}$, where the $e_i$ are positive integers and the 
$g_i(x)$ are distinct, monic, irreducible, and non-constant. We have the following
basic facts about square-free factorizations.
\medskip

%
%
\begin{prop}\label{pr33}
Fix a prime $p \ge 2$, and let $q= p^f$. Consider the set $\sF_{n, q}$ 
of all monic  polynomials in $\FF_q[X]$
of degree $n$, so that $|\sF_{n, q}|= q^n$.

(1) Exactly $q^{n-1}$ of these polynomials have discriminant $\Disc(f) =0$ in $\FF_q$.
Equivalently, exactly $q^{n-1}$ of these polynomials 
are not square-free when factored into irreducible factors
over $\FF_q[X]$.

 (2) The number $N({\ssy}; q)$ of $f(x) \in \sF_{n, q}$ whose factorization over $\FF_q$ into
 irreducible factors is square-free of degree type $\ssy:=(\ssy_1,..., \ssy_r)$,
 with $\ssy_1 \ge \ssy_2 \cdots \ge \ssy_r$ 
 having $c_i= c_i(\ssy)$ factors of degree $i$ satisfies
 \beql{221}
 N({\ssy}; q) = \prod_{i=1}^n \Big( {{M_{i}(q)}\atop{c_i(\ssy)}} \Big) = N_{\ssy}(q),
 \eeq
 in which  $N_{\ssy}(X)$ denotes the cycle polynomial for $\ssy$.
 \end{prop}

\begin{proof}
(1) This result can be found in  \cite[Prop. 2.3]{Rosen:2002}.
Another proof, due to M. Zieve, is given in \cite[Lemma 4.1]{Weiss:2013}.

(2) This result is well known, see for example
S. D. Cohen \cite[p. 256]{Cohen:1970}.
It  follows from counting all unique factorizations of the given type.
 \end{proof}

This proposition has the following consequence.

\begin{prop}\label{pr37}
Consider  a random monic polynomial $g(X)$  of degree $n$ drawn from $\FF_q[x]$ with the
uniform distribution, where $q=p^f$. Then the  probability of $g(x)$ having a factorization 
into irreducible factors of  splitting type $\ssy \in \tnss$, conditioned on $g(x)$ having a square-free factorization,
is exactly $\nu_{n, q}^{\ast}(C_{\ssy})$. That is, 
$$
\nu_{n, q}^{\ast}( C_{\ssy}) ={\rm Prob} [ g(x) ~\mbox{has splitting type} ~~{\ssy}\, | \,
g(x) ~\mbox{is square-free} ].
$$
\end{prop}


\begin{proof}
Proposition ~\ref{pr33} (1), and (2) together evaluate the conditional probability
$$
{\rm Prob} [ g(x) ~\mbox{has splitting type} ~~C_{\ssy}\, | \,
g(x) ~\mbox{is square-free} ] = \frac{1}{q^n - q^{n-1} }
\prod_{i=1}^n \Big( {{M_i(q)}\atop{c_i(\ssy)}} \Big).
$$
Comparing  the right side with the definition 
 (\ref{113a}) of the  splitting measure shows that it
equals $\nu_{n, q}^{\ast}(C_\ssy)$. 
\end{proof}


\subsection{Nonnegativity conditions for  splitting measures: Positive real $z$}\label{sec34a}

This paper is  concerned with the case that $z= t $ is a real number ($t \ne 0,1$).
In this  case the measure is real-valued, and is   a signed measure, of total (signed) mass one by Proposition \ref{pr35}.

We now treat ranges of positive real $z$ and prove Theorem \ref{th20a}, which  specifies 
several  real parameter ranges where  these measures are  nonnegative, 
and  so define  probability measures; these parameter  values include all integer values  $z= m \ge 2$.

\begin{proof}[Proof of Theorem \ref{th20a}.]
To decide on nonnegativity or positivity of $\nu_{n, t}^{\ast}(C_\ssy)$,
it suffices to study the individual terms
$\binom{M_i(t)}{c_i(\ssy)}$ for $1  \le i \le n$ 
and to show nonnegativity (resp. positivity) of each of 
the numerators 
\begin{equation}\label{lower-fac}
(M_i(t))_{c_i(\ssy)} = M_i(t) (M_{i}(t)-1) \cdots (M_i(t) - c_i(\ssy) +1).
\end{equation}

(1) We verify that  for $t > n-1$ and  all $\ssy$, all terms
 in the definition \eqref{spl-def} of $\nu_{n, t}^{\ast}(C_\ssy)$
for $1 \le i \le n$  have
$\binom{M_i(t)}{c_i(\ssy)} >0.$
The positivity of the terms in \eqref{lower-fac} is immediate for $n=1$ so suppose $n \ge 2$. 
Using  Lemma \ref{le32} (2)  for  $t > n-1$ 
we have 
$$
M_i(t) > \left\lfloor \frac{n}{i} \right\rfloor -1 \ge c_i(\ssy) -1,
$$
whence  all factors in the product \eqref{lower-fac} are positive,
as asserted.

(2) For each  integer $2 \le k \le n-1$, the normalizing factor
$\frac{1}{ k^{n-1}(k-1)}$ in the definition is positive.
Since $M_i(X)$  is an integer-valued polynomial for all $i \ge 1$, each term
in the product definition of $\nu_{n, k}^{\ast} (C_\ssy)$ is a binomial
coefficient, hence is nonnegative. This proves nonnegativity of the $k$-splitting measure.
Finally, for 
the identity conjugacy class  $C_{\langle1^n\rangle} = \{ e \}$, 
for $2 \le k \le n-1$ we have that 
$\nu_{n, k}^{\ast} (C_{\langle 1^n \rangle}) = 0,$
since in these case the $i=1$  factor  in \eqref{lower-fac}
has $(M_1(k))_{n}=0$.

(3) The limit as $t \to \infty$ is driven by the lead term asymptotics of
the polynomial $M_i(t).$ Using $\sum_{i} i c_i(\ssy) =n$ 
and $M_i(t) = \frac{1}{i} t^i +O(t^{i-1})$ we obtain
\begin{eqnarray*}
\lim_{t \to \infty} \nu_{n, t}^{\ast}(C_\ssy)
 &= & \lim_{t \to \infty}
 \prod_{i=1}^n \frac{M_i(t)^{c_i(\ssy)}}{ c_i(\ssy)!\, t^{i c_i(\ssy)}}\\
 &= & \prod_{i=1}^n \frac{1}{c_i(\ssy)! i^{c_i(\ssy)}}
  =  \frac{|C_{\ssy}|}{n!},
\end{eqnarray*}
as asserted.
\end{proof}


\subsection{Nonnegativity conditions for  splitting measures: negative real $z$}\label{sec34b}

We prove  complementary  results specifying  some negative real parameter values $z= -t <0$
where $\mu_{n, -t}(\cssy)$ is nonnegative and so defines a probability measure.

\begin{thm}\label{th20b}
Let $n \ge 2$. Then for real $z= -t <0$ the  signed measures  $\nu_{n, t}^{\ast}$ on $S_n$ have the following properties:

(1)    For all  real values  $t>0$ having $t(t+1)> n-2$,
the measure $\nu_{n, -t}^{\ast}$ on $S_n$  is strictly positive,
so that it  defines a probability measure on $S_n$ with full support.

(2)  For all  integers $k \ge 1$ having $t(t+1) \le n-2$ the measure  $\nu_{n,-k}$ is nonnegative 
and defines a probability measure on $S_n$. This measure does not have full support. 
it  is  zero on
the conjugacy class $\cssy$ with $\ssy= \langle 2^{n/2} \rangle$ if $n$ is even,
and on the  conjugacy class with  $\ssy = \langle  1, 2^{(n-1)/2}\rangle$ if $n$ is odd.

(3) There holds for all  $g \in S_n$,
$$
\lim_{t \to \infty} \nu_{n, -t}^{*}(g) = \frac{1}{n!}.
$$
\end{thm}

\begin{proof}
(1) To show positivity of the measure we keep track of the signs of all the
factors in the definition \eqref{spl-def}.
Since $t >0$ the  prefactor has  sign 
$$
\sgn \big(\frac{1}{(-t)^{n-1}(-t-1)} \big)= (-1)^n. 
$$
Lemma  \ref{le32} (3) then gives for $t(t+1) > n-2$ that
$$
M_{2j}(-t) > \left\lfloor \frac{t}{2j}\right\rfloor -1.
$$
Since $c_j(\ssy) \le  \left\lfloor  \frac{n}{j} \right\rfloor$, we obtain the positivity of all even degree terms, as
\[
(M_{2j}(-t))_{c_{2j}(\ssy)} = M_{2j}(-t) (M_{2j}(-t)-1) \cdots (M_{2j}(-t) - c_{2j}(\ssy) +1)>0.
\]
We assert that all odd degree terms have $M_{2j+1}(-t) < 0$. Assuming this
is proved, we obtain $\sgn( (M_{2j+1}(-t))_{c_{2j+1}(\ssy)}) = (-1)^{c_{2j+1}(\ssy)} = (-1)^{(2j+1)c_{2j+1}(\ssy)}$.
It follows that
$$
\sgn( \nu_{n, t}^{\ast} (\cssy) )= (-1)^n (-1)^{\sum_{i} j c_{i}(\ssy)} = (-1)^{2n} = 1,
$$
showing the required positivity. 

It remains to show that all $M_{2j+1}(-t) <0$. This holds for $t \ge 2$ by Lemma \ref{le32} (3), and
$t \ge 2$ whenever $n \ge 8$.  For the remaining cases we check\\
 $M_{1}(-t) = -t < 0$ for $t >0$,
and that for $2j+1= 3, 5, 7$ we have
$M_{2j+1}(-t) = -t^{2j+1} +t <0 $ for $t > 1$,

(2) To show nonnegativity of the measure $\nu_{n, -k}^{\ast}$ for those positive 
integer $k$ with  $k(k-1) \le n-2$,
 the argument of (1)  still applies with 
the following changes. For even indices $2j$, we use the fact that $M_{2j}(-k)$
is a positive integer, so either the descending factorial remains positive or else
is zero if a zero is encountered. So the sign of this term may be treated as positive.
For the odd indices $2j+1$, either the initial value $M_{2j+1}(-k) =0$,
in which case the measure is $0$, or else $M_{2j+1}(-k) <0$ and the sign argument
above applies. One has $M_{2j+1}(-k) <0$ if $k \le 2$ so the only problematic
value is $M_{2j+1}(-1)$. This value is always $0$, as may be checked. Thus nonnegativity
of the measure follows. 

It remains to show that the measure does not have
full support.  We verify  for $n=2\ell$ that  for $\ssy= \langle 2^{\ell} \rangle$,
 one has $\nu_{n, -k}^{\ast}(\cssy)=0$ 
for all positive integers $k$ with $k(k+1) \le n-2$. Here $m_2(\ssy)=\ell$ and the integer
\[
1 \le M_{2}(-k) = \frac{1}{2} k (k+1) \le \frac{n-2}{2} = \ell-1,
\]
 so that 
the descending factorial $(M_2(-m))_{\ell} =0.$  One verifies similarly that for $n=2\ell+1 \ge 3$
one has $\nu_{n, -k}(\cssy)=0$ for $\ssy= \langle 1, 2^{\ell} \rangle$,  where again $m_2(\ssy)=\ell$ and 
$\frac{n-2}{2} = \ell- 1.$

(3) This limit behavior follows similarly  to the case of  Theorem \ref{th20a} (3).
\end{proof}

\section{Counting $S_n$-Polynomials with Specified Splitting Types}\label{sec4}


\subsection{Counting monic $S_n$-polynomials with coefficients in a box}\label{sec41}

It is well-known that, in a suitable sense, almost all monic polynomials with
$\ZZ$ coefficients have a splitting field that is an $S_n$-extension of $\QQ$.
This was proved in 1936 by van der Waerden \cite{Waerden:1936}, who showed 
 that the fraction  of  all monic degree $n$ polynomials in $\ZZ[x]$ having 
 all coefficients in a box $|a_i| \le B$ that have  a  splitting field with Galois group $S_n$ 
 approaches  $1$ as $B \to \infty$. 
An improved  quantitative form of this assertion was given in 1973 by Gallagher \cite{Gallagher:1973}, 
which we formulate as follows.

 
 \begin{thm}[Gallagher]\label{th41}
 For integer $B \ge 1$ let 
 $\mathcal{F}_n(B)$ be the set of monic, degree $n$ polynomials in $\ZZ[x]$ with all coefficients in 
 the box $[-B+1,B]$; there are $(2B)^n$ such polynomials.
   Let $E_n(B)$ denote the proportion of polynomials
  in $\mathcal{F}_{n}(B)$ which do not have splitting field with Galois group $S_n.$ 
 Then there exists a positive constant $\alpha_n$, depending only on $n$, such that for all $B>2$, 
\begin{equation}\label{311b} 
\frac{E_n(B)}{(2B)^n} \le \alpha_n  \frac{\log{B}}{\sqrt{B}}.
\end{equation}
\end{thm}

We remark that all polynomials with coefficients in the box  $\sF_n(B)$ satisfy
\begin{equation}\label{315}
|\Disc(f)|  \le (4B)^{n(n-1)}.
\end{equation}
Indeed, we have  $\Disc(f) = \prod_{1 \le i<j\le n}(\theta_i -\theta_j)^2$, so it suffices
to show that $f(x) \in \sF_n(B)$ have all roots $|\theta_i| < 2B.$ This holds because
if some root $|\theta| \ge 2B$ then $|\theta^{n-j}| \le |\theta|^n/(2B)^j$, whence
$$
|a_{n-1} \theta^{n-1} + \cdots + a_1 \theta + a_0|  \le  B \Big (\sum_{j=1}^{n}\frac{|\theta|^n}{(2B)^j} \Big) \\
= |\theta|^n \Big ( \sum_{j=1}^n \frac{1}{2^j}\Big) < |\theta^n|.
$$
which contradicts $\theta$ being a root of $f(x).$

The error estimate  in Gallagher's estimate was recently improved by Dietmann \cite{Dietmann:2013}
 to 
 \begin{equation}
 \frac{E_n(B)}{(2B)^n} = O_{\epsilon}\left(B^{-(2-\sqrt{2}) + \epsilon} \right).
 \end{equation}
 Improvements of Gallagher's results in some other directions are given in Zywina \cite{Zywina:10}.



\subsection{Density of $S_n$-polynomials with specified splitting types}\label{sec42}

Our object  is to refine the result  above by counting the number of such polynomials
generating an $S_n$-extension that have  a given  splitting
type  at a finite set of primes.
As above, for  integer $B$ let $\sF_{n}(B)$ denote the set of monic polynomials of degree $n$
with coefficients $-B< a_i\le B,$ so that $\#\{ f(x) \in \sF_n(B)\}= (2B)^n$.
Theorem \ref{th12} is the special case $r=1$ of the following  result.

\medskip


\begin{thm}~\label{th42}
Let  $n \ge 2$ be given, and let $\sS= \{ p_1, ..., p_r\}$ be a finite set of primes and let
$\SSY=\{\ssy_1, ..., \ssy_r\}$ be a corresponding set of splitting symbols. 

(1) Let $\sF_n(B;\sS)$ denote  the set of all  polynomials $f(x)$ in $\sF_n(B)$
such that \\
$\gcd(\Disc(f), \prod_{i=1}^r p_i) = 1. $
Then
\begin{equation}\label{311a}
\lim_{B \to \infty}   \frac{\#\{ f(x) \in \sF_n(B; \sS)\}}{\#\{ f(x) \in \sF_n(B)\}}= \prod_{i=1}^r \Big(1- \frac{1}{p_i}\Big).
\end{equation}

(2) Let $\sF_n(B: \{ \sS; \SSY\})$ denote the set of all  $f(x)$ in $\sF_n(B,\sS)$
such that:  
\begin{enumerate}
\item[(i)]
 $f(x)$ has splitting field $K_f$ that is an $S_n$-extension of $\QQ$.
\item[(ii)]
 The splitting type of $f(x) ~(\bmod ~p_i)$
is $C_{\ssy_i}$  for $1 \le i \le r.$
\end{enumerate}
Then
\begin{equation}\label{312}
\lim_{ B \to \infty}
 \frac{\#\{ f(x) \in \sF_n(B; \{\sS, \SSY\})\}}{\#\{ f(x) \in \sF_n(B; \sS)\}}
 =  \prod_{i=1}^r \nu_{n, p_i}^{*}(C_{\ssy_i}).
\end{equation}
\end{thm}

We note that the condition
 $\gcd(\Disc(f), \prod_{i=1}^r p_i) = 1$  on a monic irreducible polynomial 
  guarantees that the 
field $K= \QQ(\theta)$ generated by a single root of $f(x)$
is unramified over all the primes in $\sS$.
In that case,  the discriminant $Disc(f)$ detects the discriminant of the 
ring $\sO_f = \ZZ[1, \theta, ..., \theta^{n-1}],$ which is a subring
of the full ring of integers $\sO(K)$ of the field $K = \QQ(\theta)$
generated by a root of the polynomial. We have
$$\Disc(f) = \Disc(K) [\sO(K): \sO_f]^2,$$
so that $p \nmid \Disc(f)$ implies $p \nmid \Disc(K)$.

We will derive Theorem ~\ref{th42} from two quantitative estimates given below.
We begin with an estimate for the event
$\gcd(\Disc(f), \prod_{i=1}^r p_i) = 1. $


 \begin{lemma}\label{le43} 
Let $n \ge 2$. Let $\sS= \{ p_1, p_2, ..., p_r\}$ and  $M = \prod_{i=1}^r p_i$.  Then for $B \ge 2nM$, 
\[
\left|\frac{\#\{ f(x) \in \sF_n(B; \sS)\}}{\#\{ f(x) \in \sF_n(B)\} } -\prod_{i=1}^r (1- \frac{1}{p_i})  \right| \le\, \frac{ 2nM}{B}.
\]
\end{lemma}

\begin{proof} 
For each prime $p$ the behavior of $Disc(f)$ $ (\bmod\,p)$ is determined  by \\
$(a_0, a_1, ..., a_{n-1}) ~(\bmod ~p).$ Thus if $M$ divides $B$ then 
Proposition \ref{pr33}\,(1) shows that
exactly a fraction of $\frac{1}{p}$ of these polynomials have
$\Disc(f) \equiv  0~(\bmod ~p).$  The polynomials are labelled by lattice points in
the closed box $[-B+1, B]^n$, and we call a lattice point {\em admissible} if it 
corresponds to a polynomial in $\sF_n(B, \sS)$.
For a general $B$ we first round down to a box of side $B'= M\left\lfloor \frac{B}{M} \right\rfloor$,
and using there  the Chinese remainder theorem  we find exactly
$(2B')^n\prod_{i=1}^r(1- \frac{1}{p_i}) $ admissible polynomials in the smaller 
box belong to $\sF_n(B, \sS)$.  This number undercounts 
$(2B)^n\prod_{i=1}^r (1- \frac{1}{p_i}) $
by amount
$\prod_{i=1}^r(1- \frac{1}{p_i})((2B)^n - (2B')^n)$.
Similarly we may round up to a box of side $B'' = M \lceil \frac{B}{M} \rceil$
and using there a similar argument we find exactly
$\prod_{i=1}^r(1- \frac{1}{p_i}) (2B'')^n$ admissible polynomials in the larger
box. Thus
$$
(2B')^n\prod_{i=1}^r\Big(1- \frac{1}{p_i}\Big) \le |\sF_n(B; \sS)| \le (2B'')^n\prod_{i=1}^r\Big(1- \frac{1}{p_i}\Big) 
$$
We now use the inequality, valid for real $|x| \le \frac{1}{2n}$, 
 $$
 1+ 2n|x| \ge (1+ x)^{n} \ge 1- 2n|x|.
 $$
Since  $B''- B' \le M$, 
the inequality gives for $B \ge 2nM$, 
\begin{equation}\label{333a}
(2B'')^n - (2B')^n \le
(2B)^n\Big((1 + 2n \frac{B''-B}{B}) -(1-2n\frac{B-B'}{B} ) \Big) \le (2B)^n (\frac{2nM}{B}).
\end{equation}
This  yields the estimate
\begin{equation}\label{334a}
 \#\{ f(x) \in \sF_n(B; \sS)\} =(1+ \epsilon_n(B;\sS))(2B)^n \prod_{i=1}^r (1- \frac{1}{p_i}),
\end{equation}
with
$$
|\epsilon_n(B;\sS) | \le \frac{2nM}{B}.
$$
Dividing both sides by $\#\{ f(x) \in \sF_n(B)\} = (2B)^n$ yields the desired bound.
\end{proof}

Now we derive the main  estimate from which Theorem~\ref{th42} will follow. 
\begin{thm}\label{th44}
Let $n \ge 2$. Let $\sS :=\{ p_1, p_2, ..., p_r\}$ be a finite set of primes and let
$\SSY:= \{\ssy_1, ..., \ssy_r\}$ be a set of splitting types.
Let   $\sF_n(B: \{ S; \SSY\})$ denote the set of all polynomials $f(x)$ in $\sF_n(B)$
such that:  
\begin{enumerate}
\item[(i)]
$ \gcd(\Disc(f), \prod_{i=1}^r p_i) = 1$; 
\item[(ii)]
 The splitting type of $f(x) ~(\bmod ~p_j)$
is $C_{\ssy_j}$, for $1 \le j \le r;$
\item[(iii)]
 $f(x)$ has splitting field $K_f$ that is an $S_n$-extension of $\QQ$.
\end{enumerate}
Then, setting $M= \prod_i p_i$, for $B \ge 4nM$ there holds 
$$
\mid \frac{\#\{f(x) \in \sF_n(B;  \{\sS, \SSY\})\}}{\#\{f(x) \in \sF_n(B; \sS)\}}- \prod_{i=1}^r \nu_{n, p_i}^{\ast}(C_{\ssy_i})
  \mid \,  \le \, 2\prod_{i=1}^r(1-\frac{1}{p_i})^{-1}\alpha_n \frac{\log B}{\sqrt{B}} +  \frac{4nM}{B},
$$
\end{thm}

\begin{proof}
Let $\sF_n(B, \{ \sS, \SSY\})^{+}$ denote the set of all polynomials $f(x)$ in $\sF_n(B)$
that satisfy conditions (i) and (ii) above. Theorem~\ref{th41} then gives
$$
0 \le  |\sF_n(B, \{ \sS, \SSY\})^{+}| \, - \, |\sF_n(B, \{ \sS, \SSY\})| \le (2B)^n \Big(\alpha_n \frac{\log B}{\sqrt{B}} \Big).
$$
For splitting types on  box of side $B' = M \left\lfloor \frac{B}{M} \right\rfloor$ by reduction $(\bmod~ M)$ 
together with Proposition ~\ref{pr33}  (2) and the 
Chinese remainder theorem  we get a product distribution of all splitting types $(\bmod~p_i)$ for $1 \le i \le r$,
$$
 |\sF_n(B', \{ \sS, \SSY\})^{+}| = (2B')^n \prod_{i=1}^n \frac{1}{p_{i}^n} N_{\ssy_i} (p_i^{}),
$$
where $N_{\ssy_i}(\cdot)$ is a cycle polynomial.
We have a similar formula  for an enclosing box of side $B'' = M \lceil \frac{B}{M} \rceil,$ with $(2B'')^n$ replacing $(2B')^n$.
Assuming  $B \ge 2nM$
 we obtain by an application of  (\ref{333a}) that
$$
  |\sF_n(B, \{ \sS, \SSY\})^{+}| = (1+ \epsilon_n(B; \{\sS, \SSY\})) (2B)^n\prod_{i=1}^n \frac{1}{p_{i}^n} N_{\ssy_i} (p_i^{}),
$$
with the error estimate
$$
|\epsilon_n(B, \{\sS, \SSY\})| \le \frac{2nM}{B}.
$$
Next we note that 
$\frac{1}{q^n} N_{\ssy}(q)=\left(1- \frac{1}{q}\right) \nu_{n, q}^{\ast}(C_{\ssy}).$
Substituting this for each $p_i$ in the formula above and using our original bound for $|\sF_n(B, \{ \sS, \SSY\})|$ yields
$$
\left| |\sF_n(B, \{ \sS, \SSY\})| \, - \, (2B)^n \prod_{i=1}^n \Big(1- \frac{1}{p_i}\Big) \nu_{n, p_i}^{\ast}( C_{\ssy_i})\right|
\le (2B)^n \Big( \frac{2nM}{B} \prod_{i=1}^n \Big(1- \frac{1}{p_i}\Big) + \alpha_n \frac{\log B}{\sqrt{B}} \Big).
$$
For $B \ge 2nM$, we replace $(2B)^n\prod_i(1- \frac{1}{p_i})$ with $|\sF_n(B; \sS)|$ using (\ref{334a}) 
we obtain 
$$
\left| | \sF_n(B; \{ \sS, \SSY\})| - |\sF_n(B; \sS)| \prod_{i=1}^r  \nu_{n, p_i}^{\ast}(C_{\ssy_i})\right|
 \le (2B)^n \Big( \frac{4nM}{B} \prod_{i=1}^n \Big(1- \frac{1}{p_i}\Big)
 + \alpha_n \frac{\log B}{\sqrt{B}}\Big)
$$
The result follows on dividing both sides by  
$$ |\sF_n(B; \sS)|= (1-\epsilon_n(B; \sS)(2B)^n\prod_{i=1}^n \Big(1- \frac{1}{p_i}\Big),$$
noting for $B \ge 4nM$ that (\ref{334a})  implies
$|\epsilon_n(B: \sS)| \le \frac{1}{2}.$
\end{proof} 

\noindent{\em Proof of Theorem~\ref{th42}.} 
This follows directly from Lemma~\ref{le43} and Theorem~\ref{th44} on
letting $B \to \infty.$
$~~~\Box$\\

\begin{remark}
The conclusion in Theorem \ref{th42} is  insensitive to the shape of the box bounding the coefficients
as long as the box increases homothetically as $B \to \infty$, e.g. one can use $-c_{n, j} B < a_j < c_{n, j} B$, where
$c_{n, j}$ are positive constants independent of $B$, and derive exactly
the same limiting formula. For example, $c_{n, j} = \left(  {{n}\atop{j}} \right)$ is
another natural choice.

\end{remark}


\subsection{Existence of $S_n$-number fields with specified splitting types: Proof of Theorem \ref{th13}.}\label{sec43}

We first remark  on a special property of  the symmetric group $S_n$
 as a Galois group, represented as a  permutation group acting transitively on the
 roots of a polynomial, that distinguishes  it from  some of  its subgroups. 
  Let $G$ be a 
permutation group $G \subset S_n$ (i.e. a permutation representation
of the abstract group $G$).  The elements in a
conjugacy class in $G$ necessarily have  the same cycle type as permutations, but the converse
need not hold. That is,  the cycle type of  a conjugacy class in $G$ 
need not determine it uniquely. For example, 
the group $G= \{ (12) (34), (13)(24), (14)(23), (1)(2)(3)(4)\} \subset S_4$ is abelian
so all conjugacy classes have size $1$ but three of these classes have identical cycle
structures. 
This uniqueness property does hold for cycle types for the full symmetric group $S_n$,
which has the consequence that the  cycle type  of an $S_n$ polynomial
having a  square-free factorization $(\bmod~p)$ uniquely determines the Artin symbol
for an $S_n$-number field obtained by adjoining one root of it. \medskip

\noindent{\em Proof of Theorem~\ref{th13}.} 
(1) $\Rightarrow$ (3).
By hypothesis we are given  $\sS= \{ p_1, ..., p_r\}$ and splitting types $\SSY=\{ \ssy_1, .., \ssy_r\}$  with
the property that all $\nu_{n, p_i}^{\ast}(C_{\ssy_i}) >0$. We will show the number of
$S_n$-number fields $K$ whose Galois closure $K'$ has the given Artin symbols
\begin{equation}\label{351}
\Big[ \frac{K'/\QQ}{(p_i)}\Big] = C_{\ssy_i}\,, ~~~~~~1 \le i \le r
\end{equation}
is infinite, by showing it is arbitrarily large. Since the splitting type of a polynomial
$f(X)$ generating an $S_n$-number field modulo $p$ determines the corresponding
Artin symbol, it suffices to specify factorizations of polynomials $(\bmod ~p_i)$
 which we can do using  Theorem ~\ref{th42}. 

Given $k \ge 1$ we choose
$\sS_k :=  \sS \bigcup \sS_k^{*}$ with $\sS_k^{*} =\{ p_{r+1}, \cdots, p_{r+k}\}$ being a set
of $k$ auxiliary primes that satisfy
 $n\le p_{r+1} < p_{r+2} < ...< p_{r+k}$ and disjoint from the primes in $\sS$. 
In that case we may  choose splitting symbols $\SSY_k := \{ \ssy_{r+1}, ..., \ssy_{r+k}\}$ 
arbitrarily in $S_n$ for the auxiliary primes and the condition $\nu_{n, p_{r+j}}^{*}(C_{\ssy_{r+j}})>0$
will automatically hold  by 
Theorem ~\ref{th20a}\, (2).  
The square-free condition on the
 polynomial modulo each $p_i$ guarantees that the polynomial discriminant is relatively prime to 
 $p_1p_2 \cdots p_{r+k}$ and this property guarantees that (P1) holds.
 Theorem~\ref{th42}
now implies the existence of infinitely many $S_n$-polynomials 
having the given splitting behavior at all $r+k$ primes; thus (P2) holds
for such fields. In particular there exists at 
least one such $S_n$-number field $K$ exhibiting the given splitting behavior.
Since each $S_n$ for $n \ge 2$ has  at least  two  distinct conjugacy classes,
we obtain in this way at least $2^k$ different $S_n$-number fields, all of which
match the splitting types $C_{\ssy_i}$  for $1 \le i \le r$ in (\ref{351}) and which
are distinguishable among themselves by how  the auxiliary primes $p_{r+j}$, $1 \le j \le k$
split.  Since $k$ can be arbitrarily large, the result follows.

(3) $\Rightarrow$ (2). Immediate.

(2) $\Rightarrow$ (1). By hypothesis the  given field $K$ possesses  is a monogenic order 
$\ZZ[1, \theta, ..., \theta^{n-1}]$ satisfying (P1). 
The minimal polynomial for $\theta$ is then a monic polynomial
$f(x) \in \ZZ[X]$ which satisfies $\gcd (Disc(f), p_1\cdots p_r) =1$.  This polynomial then  has square-free
factorization $(\bmod \, p_i)$ yielding the splitting types $C_{\ssy_i}$ for $1 \le i \le r$, see\footnote{The
hypothesis of Lang's Proposition 25 requires $\ZZ[1, \theta, \cdots,\theta^{n-1}]$ to be integrally closed,
i.e. the full ring of integers $O_{\kk}$. As he notes, the argument can be done by localizing over each prime ideal
$(p_i)$,  and here $(Disc(f), p_i)=1$ implies that the integral closure condition holds locally.}
Lang \cite[I. \S8, Proposition 25]{Lang:1994}.
We next observe that the  splitting type conditions are congruence conditions $(\bmod\, p_1\cdots p_n)$ on
the coefficients of $f$, and they enforce the condition
$\gcd (Disc(f), p_1 p_2 \cdots p_r) =1$.
These congruence conditions are satisfied for a positive proportion of polynomials in the box,
so in Theorem \ref{th42} the left side of \eqref{312} is positive, which certifies that
each $\nu_{n, p_i}^{\ast}(C_{\ssy_i}) >0$.
$~~~\Box$\medskip


\subsection{ Vanishing values of splitting measures: Proof of Theorem \ref{th25}}\label{sec44}

We characterize pairs $(n, p, \ssy)$ where $\nu_{n, p}^{\ast}(C_{\ssy}) =0$.

\begin{proof}[Proof of Theorem \ref{th25}]
$ (C1) \Leftrightarrow (C2)$. Since $p \ne 0, 1$ the  condition $\nu_{n, p}^{\ast}(C_{\ssy}) =0$ holds
if and only if $N_{\ssy}(p)=0.$ By Proposition \ref{pr33} (2) the latter condition holds if any only if no degree $n$ 
monic polynomial
in $\FF_p[X]$ with $Disc(f) \ne 0 \in \FF_p$ has a square-free factorization of splitting type $\ssy$.
The latter condition is exactly (C2).

$ (C1) \Leftrightarrow (C3)$. We establish the contrapositive. Suppose (C3) does not hold.
This says that   there exists a $S_n$-number field $K$ which at $(p)$ is unramified and has splitting type $\ssy$.
Now the equivalence of Theorem ~\ref{th25} applied for a single prime $p_1=p$
shows that $\nu_{n, p}^{\ast} (\ssy) >0,$ which is equivalent to the condition that  $(C1)$ does not hold. 

We remark that this argument does not establish whether or not
there exist any $S_n$ extensions  $K$ which satisfy condition (C3) for  given splitting data $\ssy$.
\end{proof}


\section{Number of $S_n$-Polynomials with Specified Splitting Types over Number Fields}\label{sec5}

We consider polynomials with coefficients drawn from an algebraic number field $\kk$,
not necessarily Galois over $\QQ$. We set $[\kk:\QQ] =d$, and
say that  an extension $L/\kk$ with $[L:\kk]=n$ is a
{\em relative $S_n$-number field} if the Galois closure $L'$ of $L$ over $k$ has $Gal(L'/\kk) \simeq S_n$.
We let $D_{\kk}$ denote the absolute
discriminant of $\kk$ over $\QQ$.

Let $\sO_{\kk}$ denote the ring of algebraic integers in $\kk$.
We consider monic polynomials
$$
f(x) = x^n + \sum_{j=0}^{n-1} \alpha_j x^j,
$$
with all $\alpha_j \in \sO_{\kk}.$
Choose an  integral basis 
$\sO_{\kk} = \ZZ[\omega_1, \omega_2, ..., \omega_d]$,
and let $\Omega=(\omega_1, ..., \omega_d)$ denote this (ordered) integral basis.
We now have
$$
\alpha_j = \sum_{k=1}^d m_{j,k} \omega_k, ~~~1 \le j \le n, 
$$
for unique $m_{i, j} \in \ZZ.$
We define $\sF_n(B; \Omega)$ to be the set of all monic degree $n$ polynomials over $\sO_{\kk}$ whose
coefficients have all $m_{i,j}$ satisfying
$-B+1 \le m_{i,j} \le B$, so there are $(2B)^{nd}$ polynomials in the box.

Next we let $\sS = \{ \fp_1, ..., \fp_r\}$ denote a finite ordered set of (distinct) prime ideals in $\sO_{\kk}$.
We allow different ideals in the list to have residue class fields of the same characteristic,
i.e. to lie over the same rational prime. We set $N_{\kk/\QQ} \fp_j = p_j^{f_j}.$
We let  $\SSY = \{ \ssy_1, ..., \ssy_r\}$ denote a finite ordered set of splitting types of $S_n$
(the different $\ssy_j$ need not be distinct).\smallskip


 

\begin{thm}~\label{th51}
Suppose that $\kk/\QQ$ is a number field, not necessarily Galois over $\QQ$, 
Let $\sS= \{ \fp_1, ..., \fp_r\}$ be an ordered finite set of distinct prime ideals in 
 $\sO_\kk$ and let $\FF_{q_i}$ denote the residue class field for $\fp_i$,
 with $q_i = N\fp_i = p_i^{f_i}$.  Suppose   $\SSY=\{\ssy_1, ..., \ssy_r\}$ is a given ordered set of splitting symbols. 
Then for fixed $n \ge 2$, the following hold.

(1) Let $\sF_n(B; \sS, \Omega)$ denote  the set of all  degree $n$ polynomials $f(x)$ in $\sF_n(B; \Omega)$
such that  $\gcd(Disc(f), \prod_{i=1}^r \fp_i) = (1),$
viewed as ideals in $\sO_{\kk}$. Then
\begin{equation}\label{411}
\lim_{B \to \infty}   \frac{\#\{ f(x) \in \sF_n(B; \sS,\Omega)\}}{\#\{ f(x) \in \sF_n(B; \Omega)\}}= \prod_{i=1}^r \Big(1- \frac{1}{q_i}\Big).
\end{equation}

(2) Let $\sF_n(B: \{ \sS; \SSY\}, \Omega)$ denote the set of all  $f(x)$ in $\sF_n(B;\sS, \Omega)$
such that:  
\begin{enumerate}
\item[(i)]
 The splitting type of $f(x) ~(\bmod ~\fp_i)$
is $C_{\ssy_i}$  for $1 \le i \le r.$
\item[(ii)]
 $f(x)$ has relative splitting field $K_f$ over $\kk$ that is an $S_n$-extension of $\kk$.

\end{enumerate}
Then
\begin{equation}\label{412}
\lim_{ B \to \infty}
 \frac{\#\{ f(x) \in \sF_n(B; \{\sS, \SSY\}), \Omega\}}{\#\{ f(x) \in \sF_n(B; \sS, \Omega)\}}
 =  \prod_{i=1}^r \nu_{n, q_i}^{*}(C_{\ssy_i}).
\end{equation}
\end{thm}

\begin{proof}
This result parallels the proof of Theorem ~\ref{th42}.  We only
sketch the details, indicating the main changes needed. Suppose $[k: \QQ] = d$.

Firstly, we have 
$$\#\{ f(x) \in \sF_n(B; \Omega)\} = (2B)^{nd}.$$
The condition for the polynomial discriminant $\gcd(Disc(f), \prod_{i=1}^r \fp_i) = (1)$
is exactly that the polynomial $f(x)$ have square-free factorization $(\bmod~\fp_i)$
for $1 \le i \le r$.
Set $M= \prod_{i=1}^r q_i = \prod_{i=1}^r (p_i)^{f_i}.$
For the limit in (1) we obtain  an exact count when going through boxes having all sides
$B=Mm$ for some integer $m \ge 1$, which is
$$
| \sF_n(B; \sS, \Omega\} |= (2B)^{nd} \prod_{i=1}^r \Big(1- \frac{1}{N\fp_i}\Big) =(2B)^{nd} \prod_{i=1}^r \Big(1- \frac{1}{q_i}\Big)
$$
For each prime ideal $\fp_i$ this holds using Proposition \ref{pr33} (1)
since we have an integral multiple of complete residue
systems $(\bmod~\fp_i)$ in the box,   and it holds for all $\fp_i$ simultaneously
using the Chinese remainder theorem for the box. Allowing a general $B$ adjusts this formula by
a multiplicative  amount $1 + O(\frac{ndM}{B})$, and letting $B \to \infty$ yields (\ref{411}).

Secondly, we introduce $\sF_{n}(B; \{ \sS, \SSY\}, \Omega)^{+}$ to be those elements of
$\sF_{n}(B; \sS, \Omega)$ that satisfy condition (i) only. We then have a bound
for the number of these $f(x)$ that do not give $S_n$-extensions of $\kk$, which is
$$
0 \le \sF_{n}(B; \{ \sS, \SSY\}, \Omega)^{+} - \sF_{n}(B; \{ \sS, \SSY\}, \Omega) \le   \alpha_n(\kk) (2B)^{nd} \frac{d\log B}{\sqrt{B^d}}.
$$
This result follows using an upper bound of Cohen \cite[Theorem 2.1]{Cohen:1981}, 
 in his result specifying  that $F_{\bf t}(x) = X^n + \sum_{i=0}^{n-1} {\bf t}_i X^i$, 
that  $K=k$, and noting the Galois group $G= S_n$ for $F_{\bf t}(X)$ over the function field
$k({\bf t}_1, \cdots, {\bf t}_n).$


Thirdly, on restricting
the box size to the special form $B=Mm$ with $m \ge 1$, one gets an exact count
$$
| \sF_n(B; \{\sS, \SSY\}, \Omega\} |= (2B)^{nd} \prod_{i=1}^r \frac{1}{ (N\fp_i)^n} N_{\ssy_i}( N\fp_i).
$$
This formula  is equivalent to 
$$
| \sF_n(B; \{\sS, \SSY\}, \Omega\} |= 
(2B)^{nd} \prod_{i=1}^r \Big(1- \frac{1}{q_i} \Big) \nu_{n, q_i}^{\ast}(C_{\ssy_i}).
$$
Changing the box size to  an arbitrary integer $B$ introduces at most a multiplicative roundoff error of
$1+ O(\frac{ndM}{B}).$ 

Fourthly, we  combine  the above estimates to obtain an analogue of Theorem~\ref{th44}, stating that
$$
\mid \frac{ | \sF_n(B;  \{\sS, \SSY\}, \Omega)|}{|  \sF_n(B; \sS, \Omega) |}
-  \prod_{i=1}^r \nu_{n, q_i}^{\ast}(C_{\ssy_i})
  \mid \,  \le \, 2\prod_{i=1}^r(1-\frac{1}{q_i})^{-1}\alpha_n(\kk)  \frac{d\log B}{\sqrt{B^d}} +  \frac{4ndM}{B}.
$$
The formula (\ref{412}) follows on letting $B \to \infty.$
\end{proof}

\begin{remark}
The conclusion in Theorem \ref{th51} is  insensitive to the shape of the box
 bounding the coefficients
as long as it is increased  homothetically as $B \to \infty$, e.g. $-c_{n, j} B < a_j < c_{n, j} B$, where
$c_{n, j}$ are positive constants independent of $B$.
\end{remark}

We next  obtain a result parallel to Theorem ~\ref{th13} on the existence of relative $S_n$-number fields $K$ over $\kk$
having prescribed splitting above a given finite set of prime ideals $\sS= \{ \fp_i: 1 \le i \le r\}$, 
and setting $ N_{\kk/\QQ} \fp_i= (q_i)$, provided  that all
the quantities $\nu_{n, q_i}^{\ast}(C_{\ssy_i})>0$. We follow
the convention that a {\em  relative $S_n$-number field $K$ over $\kk$} is a degree $n$ extension of
$\kk$ whose Galois closure over $\kk$ has Galois group $S_n$. 
We recall  that the {\em (relative) discriminant} $Disc(O_K \mid O_{\kk})$ of
any order $O$ of $K$ that contains $O_{\kk}$ 
is that ideal of $O_{\kk}$ that is generated by the
discriminants $(\alpha_1, ..., \alpha_n)$ of all the bases of $K/\kk$ which are contained in $O$.
\cite[III (2.8)]{Neukirch:1999}. The prime ideal powers dividing the relative discriminant can be
computed locally \cite[Prop. 5.7, p. 219]{Narkiewicz:1990}.\\

\begin{thm}~\label{th52}
Let $\kk /\QQ$ be a number field, not necessarily Galois over $\QQ$.
Let $\sS = \{\fp_1, ..., \fp_r\}$ denote a finite set of prime ideals of $\kk$.
 and let
$\SSY =\{ \ssy_1, ..., \ssy_r\}$ with $\ssy_j \vdash n$ be a prescribed set of splitting symbols
for these prime ideals.  Set $q_i = N_{k/\QQ} \fp_i$. Then the following conditions are equivalent.

\begin{enumerate}
\item[(1)]
The positive measure condition
$$
\nu_{n, q_i}^{\ast}(C_{\ssy_i}) >0  \,\, \mbox{for}\, \, 1 \le i \le r
$$
holds.

\item[(2)]  There exists a relative
$S_n$-number field $K/ \kk$  having the following two properties:

\begin{enumerate}
\item[(P1-$\kk$)] The field $K$ contains a monogenic order $O= O_{\kk}[1, \theta, ...\theta^{n-1}]$
whose relative discriminant $Disc(O |O_{\kk})$ is relatively prime to $\fp_1\fp_2 \cdots \fp_r.$
\item[(P2-$\kk$)]
The  Galois closure $K^{spl}$ of $K$ over $\kk$
  is unramified at all prime ideals  above those in $\sS$ and  the primes 
 in $\sS$ have  prescribed Artin symbols
$$
\Big[ \frac{K^{spl}/\kk}{(\fp_i)}\Big] = C_{\ssy_i}\,, ~~~~~~1 \le i \le r.
$$
\end{enumerate}
\item[(3)] There exist infinitely many relative $S_n$-number fields $K$ over $\kk$ having
properties {\rm(P1-$\kk$)} and  {\rm (P2-$\kk$)}.
\end{enumerate}
\end{thm}

\begin{proof} The proof parallels that of Theorem \ref{th13}, using Theorem \ref{th51} in place
of Theorem \ref{th42}.
For  $(1) \Rightarrow (3)$  we use the  fact that for a monic polynomial $f(x) \in O_{\kk}[x]$
that is  irreducible over $O_{\kk}$ one has the equality of polynomial discriminants and
relative discriminants of the associated monogenic order in $K= \kk(\theta)$, for $\theta$ a root of $f(x)$.
That is, setting $O_f:= O_{\kk}[1, \theta, \cdots, \theta^{n-1}]$.
one has   the equality
\begin{equation}\label{eq-disc}
(Disc(f)) O_{\kk} = Disc [O_{f} \mid  O_{\kk}].
\end{equation}
of $O_{\kk}$-ideals; here  $(Disc(f))$ is a principal ideal.
We use this fact to show that (P1-$\kk$) is satisfied, and apply Theorem \ref{th51} to show  (P2-$\kk$)
is satisfied. 

For $(2) \Rightarrow (1)$ the hypothesis (P1-$\kk$) with the identity \eqref{eq-disc}
implies $\fp_i \nmid (Disc(f))$ as an $O_{\kk}$-ideal and the
square-free factorization of $f(x)$ $(\bmod\, \fp_i)$ for each of the $\fp_i$.  This fact gives
the required Artin symbols $C_{\ssy_i}$, and positive density follows by Theorem \ref{th51} since all conditions
imposed are congruence conditions.
\end{proof}

To conclude the paper  we formulate a generalization of Theorem \ref{th25}. 
For  a relative extension $K/\kk$ of degree $n$  we say that a  prime ideal  $\fp$ of $O_{\kk}$ is called 
 an {\em essential relative
discriminant divisor} if it divides the relative discriminants $Disc(O| O_{\kk})$
all monogenic orders $O:=O_{\kk}]1, \theta, \cdots, \theta^{n-1}]$ of  the  field $K$ over $\kk$.

\begin{thm}\label{th53}
 Let a  number field $\kk$ together with  a prime ideal $\fp$ be given.  Let $\fp$ have
  ideal norm $N_{\kk/\QQ} (\fp) = (q) = (p^k)$.
 For a set of  splitting types $\ssy \vdash n$,  
 with $n \ge 2$, the following three conditions are equivalent.
\begin{enumerate}
\item[(C1-$\kk$)] The splitting measure at $z=q=p^k$ has
$$
\nu_{n, q}^{\ast} ( C_\ssy) = 0.
$$
\item[(C2-$\kk$)]
 For all degree $n$ monic integer polynomials $f(x)$ with coefficients in $O_{\kk}$ whose $(\bmod \,\, \fp)$ factorization
has splitting type $\ssy$,  the relative discriminant $ Disc(O_{f}| O_{\kk})$ is divisible by $\fp$.

\item[(C3-$\kk$)] All relative $S_n$-extensions $K$  of $\kk$ in which  $\fp$ is unramified and 
has splitting type $\ssy$ necessarily have $\fp$ as an essential relative discriminant divisor.
\end{enumerate}
\end{thm}
\begin{proof}
The proof parallels that of Theorem \ref{th25}.
We note only that to establish  the equivalence (C1-$\kk$) $\Leftrightarrow$ (C2-$\kk$), one uses \eqref{eq-disc}.
\end{proof}

In cases where (C1-$\kk$) holds this proof does not establish that there exist any fields satisfying (C3-$\kk$). \smallskip


\section{Generalizations}\label{sec6}


\subsection{ Characteristic polynomials of random integer matrices}\label{sec71}

The problem studied in this paper can be viewed as   a special case of   study of   characteristic polynomials
of random matrices. One may consider random matrices 
 drawn from a group like
$GL(n, \ZZ)$  with constraints on the  size of the matrix   ${\bf A}= [a_{i,j}]$ ( measured in some matrix norm), 
and also putting side conditions  on the allowed
elements.  The problem  for degree $n$ polynomials above can be encoded as such  random
$n \times n$ matrices (having entries $|a_{i,j}| \le B$)
by mapping the polynomial  $f(x)$ to the  companion
matrix having characteristic polynomial $f(x)$.  
After  reduction $(\bmod \, p)$ from $GL(n, \ZZ)$ 
one  obtains  a particular distribution of random matrices having entries over 
the  finite field $\FF_p$ with a  side condition forcing many matrix entries to be zero.
Our imposed restriction on  factorization of polynomials being squarefree corresponds
requiring that the associated matrices in $GL(n; \FF_p)$ have distinct eigenvalues,
i.e.  they belong to  semisimple conjugacy classes.
 One can ask whether  there are further interesting generalizations of the model of this paper
results in the random matrix  context.

There are many results known  considering random integer matrices  in more general models. 
In 2008 Kowalski \cite[Chap. 7]{Kowalski:2008} showed
that the characteristic polynomial of a random matrix in $SL(n, \ZZ)$ drawn using a random
walk is   an $S_n$-polynomial with probability approaching $1$ as the number of
steps increases. For splitting fields of characteristic polynomials of
random elements drawn from more general split  reductive arithmetic groups $G$ 
see  work of Gorodnik and Nevo \cite{Gorodnik-Nevo:2011},
Jouve, Kowalski and Zywina \cite{Jouve-KZ:2013}.  
In their framework the Galois group $S_n$ is replaced by the Weyl group $W(G)$ of the underlying algebraic group $G$;
 the case $W(G)= S_n$ corresponds to $G= SL_n$.  
 Lubotzky and Rosenzweig \cite{LR:2012} give a further generalization to a wider class of groups
with coefficients in a wider class of fields, where  the
``generic" Galois group of a random element may have a more complicated behavior.

There are also many results known on 
 the  distribution
of characteristic polynomials of random matrices over finite fields $\FF_q$; this subject is surveyed in 
 Fulman \cite{Fulman:2002}. 
His paper puts emphasis on $Mat(n, \FF_q)$ and $GL(n, \FF_q)$, 
and includes  results on factorization type of characteristic polynomials (see also \cite{Fulman-NP:2005}).
 Example 2 in \cite[Section 2.2]{Fulman:2002}
 observes that the  factorization type 
 for a uniformly drawn matrix in $Mat(n , \FF_q)$ has a 
distribution depending on $n$ and $q$ that approaches that of  a random degree $n$ monic polynomial in $\FF_q[X]$
as $q \to \infty$.
Fulman \cite[Section 3.1]{Fulman:2002}
 also introduces a family of probability measures $M_{GL, u, q}$ on conjugacy classes of $GL(n, \FF_q)$, which
 when conditioned on fixed $n$ do not depend on the parameter $u$ and  have
 the rational function interpolation property in the parameter $q$.
 They therefore extend to a complex parameter $z$,
 defining complex-valued measures $M_{GL, z, q}$. 
  He  remarks \cite[Section 3.3]{Fulman:2002}
 that this distribution  coincides with the
 distribution on partitions  describing the
  Jordan block structure of a random unipotent element of $GL(n,q)$.  It would be interesting to determine whether the measures
 $M_{GL, u, q}$ have  any relation to the splitting measures studied in this paper. 


\subsection{Square-free polynomials and homological stability}\label{sec72}
The   splitting measures $\nu_{n, q}^{\ast} (\cssy)$   count the relative fraction of monic square-free polynomials $(\bmod \, p)$ 
that have a given factorization type  in $\FF_q[x]$.
Recently, as a special case  of a general theory,  Church, Ellenberg and Farb \cite{CEF:2013b}  
observed that  the  monic square-free polynomials in $\FF_q[x]$ 
for $q=p^k$ label  points in an  interesting moduli space $Y_n(\FF_q)$ defined over $\FF_q$,
the {complement of the discriminant locus}, which carries an $S_n$-action.
They relate point counts on the space $Y_n(\FF_q)$ specified by factorizations of
square-free polynomials in $\FF_q[x]$ to the topology of the 
{\em configuration space} 
$$
X_n(\CC) = PConf_n(\CC) := \{ (z_1, z_2, \cdots, z_n): z_i \in \CC, z_i \ne z_j\},
$$
which itself carries an $S_n$-action.
 The configuration space  ${\rm PConf}_n(\CC)$ is an affine variety 
 which is the complement of a set of hyperplanes.
  (It is a special case of a discriminant variety, see Lehrer \cite{Lehrer:2004}.)
 Church, Ellenberg and Farb  study the $S_n$-representations produced by the $S_n$-action on the
 homology of this space and 
 show certain  homological stability properties of these representations hold
 as $n \to \infty$. They then study limiting behaviors of  
{ polynomial statistics} of these points   attached to  a fixed multivariate polynomial $P(x_1, ..., x_m) \in \QQ[x_1, ..., x_m]$
and relate this behavior to homological stability. 

The statistics they study over $Y_n(\FF_q)$ can be expressed in terms of  the $q$-splitting measures $\nu_{n, q}^{\ast}(\cdot)$,
which may permit an alternative way to view some of their results.
We hope to consider this topic  further elsewhere.

For  general results on homological stability properties 
under $S_n$-actions see  Church et al \cite{Church:2011},  \cite{CEF:2013a}, \cite{CF:2012}.\medskip

\paragraph{\bf Acknowledgments.}
The authors thank Dani Neftin for remarks on
characteristic polynomials of random matrices and for bringing
relevant references to our attention.
They thank the reviewer for helpful comments.
Some of the work of the second author  was
done   at the University of Michigan and at the Technion,
whom he thanks for support.  



\end{document}